\documentclass[11pt,english,leqno]{amsart}
\usepackage[latin1]{inputenc}
\usepackage{graphicx}  
\usepackage[all,color,matrix,arrow,curve]{xypic}
\usepackage{frcursive}
\makeatletter
\def\@tocline#1#2#3#4#5#6#7{\relax
  \ifnum #1>\c@tocdepth % then omit
  \else
    \par \addpenalty\@secpenalty\addvspace{#2}%
    \begingroup \hyphenpenalty\@M
    \@ifempty{#4}{%
      \@tempdima\csname r@tocindent\number#1\endcsname\relax
    }{%
      \@tempdima#4\relax
    }%
    \parindent\z@ \leftskip#3\relax \advance\leftskip\@tempdima\relax
    \rightskip\@pnumwidth plus4em \parfillskip-\@pnumwidth
    #5\leavevmode\hskip-\@tempdima
      \ifcase #1
       \or\or \hskip 1em \or \hskip 2em \else \hskip 3em \fi%
      #6\nobreak\relax
    \dotfill\hbox to\@pnumwidth{\@tocpagenum{#7}}\par
    \nobreak
    \endgroup
  \fi}
\makeatother

\usepackage{hyperref, endnotes}

\usepackage{mathcmd}
\usepackage{amssymb,amsmath, amscd,upgreek, amsthm}
\usepackage{mathrsfs}
\usepackage{euscript}
\usepackage[usenames,dvipsnames]{xcolor}
\numberwithin{equation}{section}
\theoremstyle{plain}{}
\newtheorem{theorem}{Theorem}[section]
\newtheorem*{thintro}{Theorem}
\newtheorem{fact}{Fact}

\newtheorem{coro}[theorem]{Corollary}
\theoremstyle{plain}
\newtheorem{prop}[theorem]{Proposition}
\newtheorem{propdefi}[theorem]{Proposition-Definition}
\newtheorem{lemma}[theorem]{Lemma}
\newtheorem*{lemmaintro}{Lemma}
\theoremstyle{remark}
\newtheorem{rema}[theorem]{Remark}
\newtheorem{defi}[theorem]{Definition}

\def\orange#1{\textcolor{Orange}{#1}}
\def\green#1{\textcolor{Green}{#1}}

\def\A{{\mathcal A}}
\def\C{{\mathbb  C}}
\def\Corps{\mathfrak F}

\def\G{{\rm G}}
\def\I{{{\tilde{\rm I}}}}

\def\KK{{\rm  K}}
\def\K{\KK}

\def\M{{\rm M}}

\def\N{{\mathbb N}}
\def\m{{ M}}
\def\mm{{\mathfrak m}}

\def\c{\circ}
\def\P{{\rm P}}

\def\R{{\rm R}}

\def\T{{\rm T}}

\def\V{{\rm V}}
\def\W{{\rm W}}
\def\X{{\rm X}}

\def\Z{{\mathbb Z}}

\def\Ap{\mathscr{A}}

\def\Aa{\mathcal{A}}
\def\Bb{\EuScript{B}}
\def\B{{\rm B}}

\def\Cute{\mathscr C}

\def\Dd{\EuScript{D}}

\def\Hh{{\tilde {\rm H}}}

\def\HhI{{\rm H}}

\def\Oo{\mathfrak{O}}
\def\OO{{\mathcal O}}
\def\Pp{\EuScript{P}}

\def\Uu{\EuScript{U}}

\def\Xx{\mathcal{X}}
\def\Zz{\mathcal{Z}}

\def\lp{\langle}
\def\rp{\rangle}

\def\Hom{{\rm Hom}}

\def\Im{{\rm Im}}
\def\t{{\mathbf \tau}}

\def\id{{\rm id}}

\def\Ind{{\rm Ind}}
\def\ind{{\rm ind}}

\def\1{{\mathbf 1}}

\def\Mod{{\rm Mod}(\H)}

\def\H{ {\tilde{\mathfrak H}}}

\def\Wf{\mathfrak W}
\newcommand\cal{\mathcal}

\def\fq{{\mathbb F}_q}

\def\kb{\mathbf{k}}

\def\val{\operatorname{\emph{val}}}
\def\Mod{{\rm Mod}}

\def\Htg{\tilde{\rm H}_\Z}

\def\k{k}

\def\Iw{{\rm  I}}

\def\Tp{{\rm T}}

\def\Gp{ { \rm G}}

\def\thg{{\Theta}}

\usepackage{comment}

 \theoremstyle{remark}%
  \newtheorem*{exa}{Example}%

\def\bf{{}_F\Bb_{F'}^+}

\def\n{N}

\usepackage{hyperref}
\definecolor{webblue}{rgb}{0, 0.7, 0.5}
\definecolor{webred}{rgb}{0.2, 0.3, 0.6} 
\hypersetup{colorlinks,  citecolor=webblue,  linkcolor=webred, urlcolor=black}

\title{Compatibility between  Satake and Bernstein-type isomorphisms in characteristic $p$}

\date{November 15, 2013}
\author{Rachel Ollivier}
\usepackage{csquotes}
\address{Columbia University, Mathematics, 2990 Broadway, New York, NY 10027}
\email{ollivier@math.columbia.edu}
\usepackage{csquotes}
\begin{document}

\maketitle

\begin{abstract} We study the center of the pro-$p$ Iwahori-Hecke  ring $\Hh_\Z$ of a connected split $p$-adic reductive group $\G$. 
 For $k$ an algebraically closed field with characteristic $p$, we 
 prove 
 that  the center of the $k$-algebra $\Hh_\Z\otimes_\Z k$  contains an affine semigroup algebra 
 which   is naturally isomorphic to the Hecke  $k$-algebra
$\cal H(\Gp, \rho)$ attached to an irreducible smooth $k$-representation $\rho$ of a given hyperspecial maximal compact subgroup of $\Gp$. 
This isomorphism is obtained using the inverse Satake isomorphism defined in \cite{invsatake}.

We apply this to  classify the simple supersingular  $\Hh_\Z\otimes_\Z k$-modules, study the supersingular block in the category of  finite length  $\Hh_\Z\otimes_\Z k$-modules,  and relate the latter to supersingular representations of $\Gp$.

\end{abstract}

\setcounter{tocdepth}{2} 

\tableofcontents

\section{Introduction}

The Iwahori-Hecke ring of a split $p$-adic reductive group $\Gp$ is  the convolution ring  of  $\Z$-valued functions with compact support in $\Iw\backslash \Gp
/\Iw$ where $\Iw$ denotes an Iwahori subgroup of $\Gp$.  It is isomorphic to the quotient of the
extended braid group  ring associated to $\Gp$ by  quadratic relations in the standard generators.
If one replaces $\Iw$ by its pro-$p$ Sylow subgroup $\I$, then one obtains    the pro-$p$ Iwahori-Hecke ring $\Hh_\Z$.
In this article we study the center of $\Hh_\Z$.
We are motivated by the smooth representation theory of $\Gp$ over an algebraically closed field $k$ with characteristic $p$ and subsequently 
will be  interested in the  $k$-algebra $\Hh_k:=\Hh_\Z\otimes_\Z k$. 
We construct   an isomorphism of $k$-algebras  between a subring of the center of $\Hh_k$ and (generalizations of) spherical Hecke $k$-algebras by means of the inverse mod $p$ Satake isomorphism  defined  in \cite{invsatake}. 
This result is the \emph{compatibility between Bernstein and Satake isomorphisms} referred to in the title of this article.
We then explore some consequences of this compatibility. In particular, we  study and relate the notions of supersingularity for Hecke modules and $k$-representations of $\G$.

\subsection{Framework and results} Let  $\Corps$ be a nonarchimedean locally compact field  with residue characteristic $p$ and $k$  an algebraic closure of the residue field.    We choose a uniformizer
$\varpi$. %We consider the group $\Gp={\rm GL}_3(\Corps)$  and its smooth representations with c\oe fficients in an algebraically closed field with caracteristic $p$. We will denote by
Let $\Gp := \mathbf{G}(\mathfrak{F})$ be the group of $\mathfrak{F}$-rational points of a connected reductive group $\mathbf{G}$ over $\mathfrak{F}$ which we assume to be $\mathfrak{F}$-split. 
 In  the semisimple  building $\mathscr {X}$ of $\Gp$, we choose and  fix a chamber $C$
 which amounts to choosing an Iwahori subgroup $\Iw$ in $\Gp$, and we denote by
$\I$ the pro-$p$ Sylow subgroup of $\Iw$.  The choice  of $C$ is unique up to conjugacy by an element of $\Gp$. We consider the associated pro-$p$ Iwahori-Hecke ring $\Hh_\Z:=\Z[\I\backslash \Gp/ \I]$
  of  $\Z$-valued functions with compact support in $\I\backslash \Gp
/\I$ under convolution.

\medskip
Since $\G$ is split, $C$ has at least one  hyperspecial vertex    $x_0$ and we denote by $\K$ the  associated maximal compact subgroup of $\Gp$.
Fix a maximal $\Corps$-split torus $\T$ in $\Gp$ such that
the corresponding  apartment $\Ap$ in $\mathscr{X}$ contains $C$.  
The set  $\X_*(\T)$ of  cocharacters of $\T$ is naturally equipped with an action of  the finite Weyl group $\Wf$. 
The choice of $x_0$ and of $C$ induces a natural
 choice of a positive Weyl chamber of $\Ap$ that is to say of a semigroup $\X^+_*(\T)$ of dominant
cocharacters  of $\T$.

\subsubsection{\label{classi} The complex case} The structure of  the spherical algebra $\C[\K\backslash \Gp/ \K]$ of  complex functions compactly supported  on $\K\backslash \Gp/ \K$
  is understood 
  thanks to 
 the classical Satake isomosphism (\cite{Sat}, see also \cite{Gross}, \cite{Haines})
$$s: \C[\K\backslash \Gp/\K]\overset{\simeq}\longrightarrow (\C[\X_*(\T)])^\Wf.$$
On the other hand, 
the complex
Iwahori-Hecke algebra $\HhI_\C:=\C[\Iw\backslash \Gp/ \Iw]$ of  complex functions compactly supported on  $\Iw\backslash \Gp/ \Iw$ contains
a large commutative subalgebra $\Aa_\C$ defined as  the image of the \emph{Bernstein map}  $\theta: \C[\X_*(\Tp)]\hookrightarrow \HhI_\C$  which depends
on the choice of the dominant Weyl chamber (see \cite[3.2]{Lu}).
The algebra $\HhI_{\C}$  is free of finite rank over $\Aa_\C$ and 
its center $\Zz(\HhI_{\C})$   is contained in $\Aa_\C$. Furthermore, the map $\theta$  yields  an isomorphism 
$$b: (\C[\X_*(\T)])^\Wf\overset{\simeq}\longrightarrow \cal Z(\HhI_{\C}).$$
This was proved  by Bernstein     (\cite[3.5]{Lu}, see also \cite[Theorem 2.3]{Haines}).
By  \cite[Corollary 3.1]{Dat} and  \cite[Proposition 10.1]{Haines},
 the  \emph{Bernstein isomophism} $b$
is compatible with $s$ in the sense that 
the composition $(e_\K\star.) b$ is an inverse for $s$, where $(e_\K\star.) $ denotes the convolution  by the characteristic function of $\K$.

\medskip

\subsubsection{Bernstein and Satake isomorphisms in characteristic  $p$}  After defining an integral version of the complex Bernstein map, Vignéras gave in \cite{Ann}
a basis for the center of  $\Hh_\Z$ and proved that $\Hh_\Z$ is noetherian and finitely generated over its center.  In the  first section of this article, we define a subring  $\Zz^\c(\Hh_\Z)$ of the center of $\Hh_\Z$   over which $\Hh_\Z$ is still finitely generated.
In Proposition \ref{prop:conjugacy} we prove that
$\Zz^\c(\Hh_\Z)$ is not affected by the choice of another apartment containing $C$ and of another hyperspecial vertex of $C$ as long as it is conjugate to $x_0$. In particular, if $\mathbf G$ is of adjoint type or $\mathbf G={\rm GL}_n$, then
$\Zz^\c(\Hh_\Z)$ depends only on the choice of the uniformizer $\varpi$.  

\medskip

The  natural image of  $\Zz^\c(\Hh_\Z)$ in $\Hh_k=\Hh_\Z\otimes_\Z k$ is denoted by $\Zz^\c(\Hh_k)$  and we prove that  it has a structure of  affine semigroup algebra. More precisely,
  we have an isomorphism of $k$-algebras (Proposition \ref{prop:Zmap})
   \begin{equation}\label{bisom}
 k[\X_*^+(\T)]\overset{\sim}\longrightarrow   \Zz^\c(\Hh_k) \subseteq \Hh_k.
 \end{equation} 
%It is  the isomorphism \emph{à la Bernstein} referred to in the title of this article. This denomination will be justified in \ref{classi}.
By the main theorem in  \cite{satake} (and in \cite{invsatake}),  this 
makes $\Zz^\c(\Hh_k)$ isomorphic to  the
 algebra ${\cal H}(\Gp, \rho)$ of  any irreducible smooth $k$-representation $\rho$ of  $\K$.  Note that when $\rho$ is the $k$-valued trivial representation $\mathbf 1_\K$ of $\K$, ones retrieves the convolution algebra $k[\K\backslash \Gp/ \K]={\cal H}(\Gp, \mathbf 1_\K)$.
  
 % which is why we  call  $\Zz^\c(\Hh_k)$  the \emph{spherical center} of $\Hh_k$ (relative to the choice of $\varpi$ and $x_0$).

\medskip

%In \cite{satake} indeed, Herzig 
%constructs  an isomorphism $\EuScript S: {\cal H}(\Gp, \rho)\overset{\simeq}\longrightarrow k[\X_*^+(\T)]$. 

In \cite{invsatake}, we constructed an isomorphism\begin{equation}
\label{monsatake}\EuScript T:  k[\X_*^+(\T)]\overset{\simeq}\longrightarrow {\cal H}(\Gp, \rho).\end{equation} 
%Under the hypothesis that  the derived  subgroup of $\mathbf G$ is  simply connected and  that $\Corps$ is a finite extention of $\mathbb Q_p${\color{magenta}
In the current article,  we prove the following theorem. 
\begin{theorem}[Theorem \ref{coro:diag}]\label{theointro}
We have a commutative diagram  of isomorphisms of $k$-algebras
  \begin{equation}\label{diag:intro}\begin{CD}k[\X_*^+(\T)] @>{\eqref{bisom}}>>  \Zz^\c(\Hh_k)\\
@|   @VV{ }V \\
k[\X_*^+(\T)] @>{\EuScript{T}}>>        \cal H(\Gp, \rho)\\
\end{CD}\end{equation}
where the vertical arrow on the right hand side is the natural morphism  of  $k$-algebras  \eqref{satakecenter} described in Section \ref{sec:compa}.
\end{theorem}

The isomorphism $\EuScript T$ was
 constructed in \cite{invsatake} by means of generalized 
  integral Bernstein maps, as are the subring $\Zz^\c(\Hh_k)$ and the
map \eqref{bisom} in the current article. By analogy with the complex case, we can see the map \eqref{bisom}  as an  isomorphism \`a  la Bernstein in characteristic $p$.
The  above commutative diagram can  then be interpreted  as a statement of compatibility between  Satake  and Bernstein isomorphisms in characteristic $p$.
 Note that under the hypothesis that
the derived subgroup of $\mathbf G$ is simply connected,
it is proved in  \cite{invsatake} that
$\EuScript T$ is  the inverse of the mod $p$ Satake isomorphism defined by Herzig in \cite{satake}. (The  extra hypothesis on $\mathbf G$ is probably not necessary).

\medskip

\medskip

If we  worked with the Iwahori-Hecke algebra $k[\Iw\backslash \Gp/\Iw]$, the analog  of $\Zz^\c(\Hh_k)$ would actually be the whole center of $k[\Iw\backslash \Gp/\Iw]$. We prove:
\begin{thintro}[Theorem \ref{centerIwa}]
The center of the Iwahori-Hecke $k$-algebra $k[\Iw\backslash \Gp/\Iw]$   is isomorphic to $k[\X_*^+(\T)]$.
\end{thintro}

\subsubsection{Generalized integral Bernstein  maps\label{gibm}}
One  ingredient of the construction of $\EuScript T$ 
in \cite{invsatake}  and of the proof of Theorem \ref{theointro} is the definition of $\Z$-linear injective maps
$$ \Bb_F^\sigma: \Z[\tilde\X_*(\T)]\rightarrow \Hh_\Z$$ defined on the group ring of the (extended) cocharacters $\tilde\X_*(\T)$, and which are multiplicative when restricted to the semigroup ring of any chosen Weyl chamber of $\tilde\X_*(\T)$ (see \ref{tilda} for the definition of $\tilde\X_*(\T)$).
The image of $\Bb_F^\sigma$ happens to be a commutative subring of $\Hh_\Z$ which we denote by $\Aa_F^\sigma$.
The parameter $\sigma$ is a sign and $F$ is a standard facet, meaning a facet  of $C$ containing $x_0$ in its closure. 
The choice of $F$ corresponds to the choice of a Weyl chamber in $\Ap$: for example if $F=C$ (resp. $x_0$), then the corresponding Weyl chamber is the  dominant (resp. antidominant) one.

The  maps $\Bb_F^\sigma$ are called  \emph{integral Bernstein maps} because they are  generalizations of the  Bernstein map $\theta$ mentioned in \ref{classi}.
In the complex case, it is customary to consider  either $\theta$ which is constructed using the dominant chamber, or $\theta^-$ which is constructed using the antidominant chamber (see the dicussion in the introduction of \cite{formulae} for example). By a result by Bernstein (\cite{LuSing}), a basis for the center of $\HhI_\C$ is given by  the central Bernstein functions $$\sum_{\lambda'\in \OO} \theta(\lambda')$$ where $\OO$ ranges over the $\Wf$-orbits in $\X_*(\T)$.  We refer to \cite{Haines} for the geometric interpretation of these functions.
 It is natural to ask whether using $\theta^-$ instead of $\theta$ in the previous formula yields the same central element in $\HhI_\C$. The answer is yes (see \cite[2.2.2]{formulae}). The proof is based on \cite[Corollary 8.8]{LuSing} and  relies on the combinatorics of  the Kazhdan-Lusztig polynomials. Note that there is no  theory of Kazhdan-Lusztig polynomials for the complex pro-$p$ Iwahori-Hecke algebra.

 Integral  (and pro-$p$) versions of $\theta$ and $\theta^-$ 
for  the ring $\Hh_\Z$ were defined in \cite{Ann}. In our language they correspond respectively to $\Bb_C^+=\Bb_{x_0}^-$ and $\Bb_{x_0}^+= \Bb_C^-$.
It is also proved in \cite{Ann} that a $\Z$-basis for  the center of $\Hh_\Z$
is given by 
\begin{equation}\label{ce}\sum_{\lambda'\in \OO} \Bb_C^+(\lambda')\end{equation}where $\OO$ ranges over the $\Wf$-orbits in $\tilde\X_*(\T)$. It is now natural to ask  whether the element \eqref{ce} is the same if 
  \textbf{a/}  we use $-$ instead of $+$, and if more generally, \textbf{b/}  we use any standard facet $F$ instead of $C$, and any sign $\sigma$. We prove:

\begin{lemmaintro}[Lemma \ref{theoA}]
The element \begin{equation*}
\sum_{\lambda'\in \OO} \Bb_{F}^\sigma(\lambda')\end{equation*} in $\Hh_\Z$ does not depend on the choice of the standard facet $F$ and of the sign $\sigma$. 

\end{lemmaintro}

\medskip
To prove the lemma, we first answer positively Question \textbf{a/} above; we then  study and exploit the behavior of the   integral Bernstein maps
upon a process of parabolic induction.  
In passing we  also consider Question  \textbf{a/} in the $k$-algebra $\Hh_k$ in the case when $\G$ is semisimple, and we suggest a link between such questions and the duality for finite length $\Hh_k$-modules defined in \cite{OS} (see Proposition \ref{theo:indep}).

\subsubsection{} 
In  Section \ref{sec:supersing}, we define and study a natural topology on $\Hh_k$ which  depends only on the conjugacy class of $x_0$. It is the $\mathfrak I$-adic topology where $\mathfrak I$ is a natural monomial ideal of
the  affine semigroup algebra $\Zz^\c(\Hh_k)$.

We define the supersingular block of the category of finite length  $\Hh_k$-modules to be the full subcategory of the modules that  are continuous for the $\mathfrak I$-adic topology on $\Hh_k$ (Proposition-Definition \ref{propdefi}).
A finite
 length $\Hh_k$-module  then turns out to be  in the supersingular block if and only if all its irreducible constituents are supersingular in the sense of 
\cite{Ann}.

In the case when the root system of $\G$ is irreducible,  we establish the following results.
We classify the simple supersingular $\Hh_k$-modules (Theorem \ref{theo:charaff} and subsequent Corollary).  (For example, when $\G$ is semisimple  simply connected,   the simple supersingular modules all have dimension $1$.) We prove in passing that even if the ideal $\mathfrak I$ does depend on the choices made, the supersingular block  is independent of all the choices.

 \medskip
 
% Kottwitz Rapoport GL(n) GSP2n
Theorem   \ref{theo:charaff} extends    \cite[Theorem 5]{Ann}-\cite[Theorem 7.3]{Oparab}
that dealt with the case of ${\rm GL}_n$ and relied on  explicit \emph{minimal expressions} for certain Bernstein functions associated to the minuscule coweights. The results of  \cite{Ann}  and \cite{Oparab} together proved
a \enquote{numerical Langlands correspondence for Hecke modules} of ${\rm GL}_n(\mathfrak F)$: there is a bijection between the finite set of all simple $n$-dimensional supersingular $\Hh_k$-modules  and the finite set of all irreducible $n$-dimensional smooth $k$-representations of the absolute Galois group of $\mathfrak F$, where the action of the uniformizer $\varpi$ on the Hecke modules and the determinant of the Frobenius on the Galois representations are fixed. Recently, Grosse-Kl\" onne constructed a functor from the category of finite length $\Hh_k$-modules 
for ${\rm GL}_n(\mathbb Q_p)$ to the category of \'etale $(\varphi, \Gamma)$-modules. This functor induces a bijection between the two finite sets above, turning the  \enquote{numerical} correspondence into a natural and explicit correspondence in the case of  ${\rm GL}_n(\mathbb Q_p)$. In fact,  Grosse-Kl\" onne constructs  such a functor (with values in a category of modified \'etale $(\varphi, \Gamma)$-modules) in the case of a general split group over  $\mathbb Q_p$ (\cite{GK2}). In the case of ${\rm SL}_n(\mathfrak F)$, Koziol has defined packets of simple supersingular $\Hh_k$-modules and built a bijection between the set of packets and a certain set of projective $k$-representations of the absolute Galois group of $\mathfrak F$;  if 
$\mathfrak F=\mathbb Q_p$, this bijection is proved to be compatible with Grosse-Kl\" onne's functor  and therefore with the explicit Langlands-type correspondence for Hecke modules of ${\rm GL}_n(\mathbb Q_p)$.  This result (\cite{K}) is a first step towards  a mod $p$ principle of functoriality  for Hecke modules.  

The current article provides, in the case of a general split group,  a  classification of the objects that one  wants to apply Grosse-Kl\"onne's functor to, in order to investigate the possibility of a Langlands-type  correspondence for Hecke modules in general.

%Let $\Hh_k^{aff}$ denote  the affine subalgebra   of the functions with support in the subgroup of $\G$ generated by all parahoric subgroups.  We describe in ??? the characters of $\Hh_k^{aff}$. We call such a character non degenerate if  it is different from the character sign and the trivial character of $\Hh_k^{aff}$ (and from certain twists of those characters).

%We prove  (Theorem \ref{theo:charaff}) that an irreducible $\Hh_k$-module is in the supersingular block if and only if it contains a character  for tThis extends \cite[Theorem 7.3]{Oparab} and  \cite[Theorem 5]{Ann}
%that dealt with the case of ${\rm GL}_n$: 
%In the case of ${\rm GL}_n$, the 
%converse statement  is   true  by \cite[Theorem 5]{Ann}.
%Therefore,  in the case of  $\mathbf G={\rm GL}_n$ as well as in the adjoint type case, the notion of supersingularity for finite length modules is  independent of all the choices.

\subsubsection{}  In \ref{onS} we consider
 an admissible irreducible smooth $k$-representation $\uppi$ of $\Gp$. In the case  where the derived subgroup of $\mathbf G$ is simply connected, we use the fact 
that
\eqref{monsatake} is  the inverse of the mod $p$ Satake isomorphism to prove that 
if $\uppi$ is supersingular, then  \begin{equation}\label{condi}\textrm{$\uppi$  is a quotient of 
$\ind_\I^\G 1/\mathfrak I \,\ind_\I^\G 1$.}\end{equation}
The  condition \eqref{condi} is equivalent to saying that $\uppi^{\I}$ contains an irreducible supersingular $\Hh_k$-module.

When  $\G={\rm GL}_n(\mathfrak F)$ and $\Corps$ is a finite extension of $\mathbb Q_p$,
we use the classification of  the nonsupersingular representations obtained in  \cite{Parabind}, the work on generalized special representations in \cite{GK}, and our Lemma \ref{theoA},  to prove that the condition \eqref{condi} is in fact a characterization of the  supersingular representations (Theorem \ref{equiv}). 

Finally, we comment  in \ref{onS} on the
 generalization of this characterization to the case of a split group (with simply connected derived subgroup), and 
 on the independence of the
characterization of the choices made.

We raise the question of  the possibility of 
 a direct proof of this characterization  that does not use the classification of the nonsupersingular representations.

\medskip

\subsection{Notation and preliminaries\label{nots}} We choose the valuation $\val_{\mathfrak{F}}$ on   $\mathfrak{F}$  normalized by $\val_{\mathfrak{F}}(\varpi)=1$ where $\varpi$ is the chosen uniformizer. 
 The  ring of integers of $\Corps$ is denoted by $\Oo$ and its residue field    by $\mathbb F_q$ where $q$  is a power of  the prime number $p$. Recall that $k$ denotes an algebraic closure of $\mathbb F_q$.
Let $\mathbf{G}_{x_0}$ and $\mathbf{G}_C$ denote the Bruhat-Tits group schemes over $\mathfrak{O}$ whose $\mathfrak{O}$-valued points are $\K$ and $\Iw$ respectively. 
   Their reductions over the residue field $\mathbb{F}_q$ are denoted by $\overline{\mathbf{G}}_{x_0}$ and $\overline{\mathbf{G}}_C$. 
   Note that $\Gp= \mathbf{G}_{x_0}(\Corps)=\mathbf{G}_{C}(\Corps)$.
   By \cite[3.4.2, 3.7 and 3.8]{Tit}, %$\overline{\mathbf{G}}_{x_0}$ is connected reductive and $\mathbb{F}_q$-split and  $\overline{\mathbf{G}}_C$ is connected. Hence 
$\overline{\mathbf{G}}_{x_0}$ is connected  reductive and $\mathbb F_q$-split.
Therefore we have ${\mathbf{G}}_C^\circ(\Oo)={\mathbf{G}}_C(\Oo) = \Iw$ and ${\mathbf{G}}_{x_0}^\circ(\Oo)={\mathbf{G}}_{x_0}(\Oo) = \K$.   Denote    by $\K_1$ the pro-unipotent radical of $\K$.   
The quotient $\K/\K_1$ is isomorphic to  $\overline{\mathbf G}_{x_0}(\mathbb F_q)$. 
The Iwahori subgroup $\Iw$  is the preimage in $\K$ of the $\mathbb F_q$-rational points of  a Borel subgroup $\overline{\mathbf B}$
with Levi decomposition $\overline{\mathbf B}= \overline{\mathbf T}\,\overline{\mathbf N}$. 
The pro-$p$ Iwahori subgroup $\I$ is the preimage in $\Iw$ of $\overline{\mathbf N}(\mathbb F_q)$. The preimage of  $ \overline{\mathbf{T}}(\mathbb F_q)$ is  the maximal compact subgroup $\Tp^0$ of $\Tp$.
Note that $\T^0/\T^1=\Iw/\I =  \overline{\mathbf{T}}(\mathbb F_q)$ where  $\T^1:= \T^0\cap\I$.

\subsubsection{Affine root datum} To the choice of $\T$ is attached  the root datum $(\Phi, {\rm X}^*(\Tp), \check\Phi, {\rm X}_*(\Tp))$.
This root system is reduced because  the group $\mathbf{G}$ is $\mathfrak F$-split. We denote by $\Wf$ the finite Weyl group  $N_\G(\T)/\T$, quotient by $\T$ of the normalizer of $\T$.   Recall that  $\Ap$ denotes the apartment of the semisimple building attached to  $\Tp$ (\cite{Tit} and \cite[I.1]{SS}, and we follow the notations of \cite[2.2]{invsatake}). We denote by $\lp\,. \,,\, .\,\rp$ the perfect pairing ${\rm X}_*(\Tp)\times {\rm X}^*(\Tp)\rightarrow \Z$. We will call coweights the elements in $\X_*(\T)$.
We identify $\X_*(\T)$ with the subgroup  $\T/\T^0$ of the extended Weyl group $\W= N_\G(\T)/\T^0$ as in \cite[I.1]{Tit} and \cite[I.1]{SS}: to an element $g\in \Tp$ corresponds a vector $\nu(g)\in \mathbb R\otimes _{\mathbb Z}\X_*({\Tp})$ defined by
\begin{equation}\label{normalization}
    \lp\nu(g),\, \chi\rp =-\val_{\mathfrak F}(\chi(g))  \qquad \textrm{for any } \chi\in \X^*(\Tp).
\end{equation} and $\nu$ induces the required isomorphism   $\Tp/\T^0\cong\X_*(\Tp)$. The group $\Tp/\T^0$ acts by translation on $\Ap$ via $\nu$.  
The actions of $\Wf$ and $\Tp/\T^0$ combine into an action of 
 $\W$  on $\Ap$ as recalled in \cite[page 102]{SS}. Since $x_0$ is a special vertex of the building,   $\W$  is isomorphic to the semidirect product $\Wf\ltimes \X_*(\T)$  where we see $\mathfrak W$ as the fixator in $\W$ of any point lifting $x_0$ in the extended apartment (\cite[1.9]{Tit}).
A coweight $\lambda$ will sometimes be denoted by $e^\lambda$ to underline that we see it as an element in $\W$, meaning as a translation on $\Ap$.
 
 \medskip

 Denote by $\Phi_{aff}$ the set of affine roots. The choice of the chamber $C$ implies  in particular  the choice of the positive affine roots $\Phi_{aff}^+$  taking nonnegative values on $C$. The choice of $x_0$ as an origin of $\Aa$ implies that we identify  the affine roots taking value zero at $x_0$ with $\Phi$. We set $\Phi^+:=\Phi_{aff}^+\cap \Phi$ and $\Phi^-=-\Phi^+$.
The affine roots can be described the following way:   ${\Phi_{aff}}=\Phi\times \Z={\Phi_{aff}^+}\coprod {\Phi_{aff}^-}$ where
$${\Phi_{aff}^+}:=\{(\alpha, r),\: \alpha\in\Phi, \,r>0\}\cup\{(\alpha,0),\, \alpha\in\Phi^+\}.$$

Let  $\Pi$ be the basis for $\Phi^+$: it is the set of simple roots. The  finite Weyl group $\Wf$ is a Coxeter system with generating set $S:=\{s_\alpha, \,\alpha\in \Pi\}$ where $s_\alpha$ denotes the (simple) reflection at the hyperplane $\lp\,.\,,\alpha\rp=0$.
Denote by $\preceq$ the partial ordering on $\X_*^+(\T)$ associated to $\Pi$.
Let  $\Pi_m$  be the set of roots in $\Phi$ that are minimal elements  for $\preceq$. 
Define the set of simple affine roots by  $\Pi_{aff}:=\{(\alpha, 0),\: \alpha\in\Pi\}\cup\{(\alpha,1),\, \alpha\in\Pi_m\}$. Identifying $\alpha$ with $(\alpha,0)$, we consider $\Pi$ a subset of
$\Pi_{aff}$.  For $A\in  \Pi_{aff}$, denote by $s_A$ the following associated reflection: $s_A=s_\alpha$ if $A=(\alpha, 0)$ and  $s_A=s_\alpha e^{\check\alpha}$ if $A=(\alpha,1)$.  
The action of $\W$ on the coweights induces an action on the  set of affine roots: $\W$ acts on $\Phi_{aff}$ by $we^\lambda: (\alpha, r)\mapsto (w\alpha,\,\, r-\lp \lambda, \alpha\rp)$ where we denote by $(w,\alpha)\mapsto w\alpha$ the natural action of $\Wf$ on $\Phi$. The length on the Coxeter system $(\Wf, S)$ extends to $\W$ in such a way that  the length $\ell(w)$ of $w\in \W$  
 is the number of affine roots   $A\in{\Phi_{aff}^+}$ such that
$w(A)\in { \Phi_{aff}^-}$. It satisfies  the following formula, for  $A \in \Pi_{aff}$ and $w\in \W$:
\begin{equation}\label{add}
   \ell(w s_A)= 
   \begin{cases}
       \ell(w)+1 & \textrm{ if }w (A)\in {\Phi_{aff}^+},\\  \ell(w)-1 & \textrm{ if }w (A)\in {\Phi_{aff}^-}.
    \end{cases}
\end{equation}

The affine Weyl group is defined as the subgroup
$\W_{aff}$ of $\W$ generated by $S_{aff} := \{ s_A , \:\: A \in \Pi_{aff}\}$. The length function $\ell$ restricted to $\W_{aff}$ coincides with the length function of the  Coxeter system  $(\W_{aff}, S_{aff})$  (\cite[V.3.2 Thm\ 1(i)]{Bki-LA}). Recall  (\cite[1.5]{Lu}) that $\W_{aff}$ is a normal subgroup of $\W$:   the set $\Omega $ of elements  with length zero  is an abelian subgroup of $\W$ and 
 $\W$ is the semidirect product $\W= \Omega \ltimes \W_{aff}$.
 The length $\ell$ is constant on the double cosets of $\W$ mod $\Omega$. In particular $\Omega$ normalizes $S_{aff}$.

\medskip

The extended Weyl group  $\W$ is equipped with  a partial order $\leq$ that extends the  Bruhat order on $\W_{aff}$.
 By definition, given $w= \omega w_{aff}$,  $w= \omega' w_{aff}' \in \Omega \ltimes\W_{aff}$, we have
$w\leq w'$ if $\omega=\omega'$ and $w_{aff}\leq w_{aff}'$ in the Bruhat order on $\W_{aff}$ (see for example \cite[2.1]{Haines}).

 \medskip
 
We fix a lift $\hat w\in N_\G(\T)$ for any $w\in  \W$. By Bruhat decomposition,  $\Gp$ is the disjoint union of all $\Iw \hat w \Iw$ for $w\in\W$.
 
\subsubsection{Orientation character\label{orientation}} 
The stabilizer of the chamber $C$ in $\W$ is $\Omega$.
We define as in \cite[3.1]{OS} the orientation character $\epsilon_C:\Omega\rightarrow\{\pm 1\}$ of $C$ by setting   $\epsilon_C(\omega) = +1$, resp.\ $-1$, if $\omega$ preserves, resp.\ reverses, a given orientation of $C$. 
Since $ \W/\W_{aff}=\Omega$ we can see $\epsilon_C$ as a character of $\W$ trivial on $\W_{aff}$. 
By definition of the Bruhat order  on $\W$, we have  $\epsilon_{C}(w)=\epsilon_{C}(w')$ for $w,w'\in \W$ satisfying $w\leq w'$.

On the other hand,  the extended Weyl group acts by affine isometries on the Euclidean space $\Ap$. We therefore have a determinant map
$\det:\W\rightarrow \{\pm 1\}$ which is trivial on $\X_*(\T)$. An orientation of $C$ is a choice of a cyclic ordering of its set of vertices  (in the geometric realization of $\Ap$). Therefore,   $\det(\omega)$ is the signature of the permutation of the vertices of $C$  induced by $\omega\in \Omega$ and
 $\det(\omega)=\epsilon_C(\omega)$.

\begin{lemma}
\begin{itemize}
\item[i.] For $w\in \W_{aff}$, we have $\det(w)= (-1)^{\ell(w)}$.

\item[ii.] For  $\lambda\in \X_*(\T)$, we have
$\epsilon_C(w)=(-1)^{\ell(e^\lambda)}$ for any $w\in\W$ such that $w\leq e^\lambda$.
\end{itemize}
\label{calcu1}
\end{lemma}

\begin{proof}
The first point comes from the fact that $\det(s)=-1$ for $s\in S_{aff}$.
For the second one,  by definition of the Bruhat order, it is enough to prove that $\epsilon_C(e^\lambda)=(-1)^{\ell(e^\lambda)}$ for 
$\lambda\in \X_*(\T)$. Decompose $e^\lambda= \omega w_{aff}$ with
$w\in\W_{aff}$ and $\omega\in \Omega$. Recall that $\omega$ has length zero.
Since $\epsilon_C$ is trivial on $\W_{aff}$, we have
$\epsilon_C(e^\lambda)=\epsilon_C(\omega)=\det(\omega)$. Since
$\det(e^\lambda)=1$ we have
$\det(\omega)=\det(w_{aff})=(-1)^{\ell(w_{aff})}=(-1)^{\ell(e^\lambda)}$.

\end{proof}
\subsubsection{Distinguished cosets representatives}

\begin{prop}

\phantomsection
%\label{representants}
\begin{itemize}
\item[i.] The set $\Dd$ of all elements $d\in \W$ satisfying
$    d^{-1}(\Phi^+)\subset {\Phi}_{aff}^+$ is a system of representatives of the right cosets $\Wf\backslash\W$. It satisfies
\begin{equation}\label{additive0}
\ell( wd)=\ell(w)+\ell(d)
\textrm{ for any $w\in \Wf$ and $d\in \Dd$.}\end{equation} In particular, $d$ is the unique element with minimal length in $\Wf d$.
\item[ii.] An element  $d\in \Dd$ can be written uniquely $d=e^{\lambda}w$ with $\lambda\in \X_*^+(\T)$ and  $w\in \Wf$. We then have $\ell(e^{\lambda})=\ell(d)+ \ell(w^{-1})=\ell(d)+ \ell(w)$.   
\item[iii.] For  $s\in S_{aff}$ and $d\in \Dd$, we are in one of the following situations:
\begin{itemize}
\item[$\bullet$] $\ell(ds)=\ell(d)-1$ in which case $ds\in \Dd$.
\item[$\bullet$] $\ell(ds)=\ell(d)+1$ in which case either $ds\in \Dd$ or $ds\in \Wf d$.
\end{itemize}
\end{itemize}
\label{prop:dist}
\end{prop}

\begin{proof} This proposition is proved in \cite[Lemma 2.6, Prop. 2.7]{Oparab} in the case of $\Gp={\rm GL}_n(\Corps)$.
It is checked in \cite[Prop. 4.6]{OS}  that it remains valid for a general split reductive group (see also \cite[Prop. 2.2]{invsatake} for ii),  except for point iii when  $s\in S_{aff}-S$. We check here that the argument goes through. 
Let $s\in S_{aff}$ and  $A$ the corresponding affine root.
Let $d\in \Dd$ and suppose that $ds\not\in \Dd$, then  there is $\beta\in \Pi$ such that
$(ds)^{-1} \beta\in \Phi_{aff}^-$ while $d^{-1} \beta\in \Phi_{aff}^+$. It implies that
$d^{-1}\beta=A$ which in particular ensures that $dA\in \Phi_{aff}^+$ and therefore $\ell(ds)= \ell(d)+1$. Furthermore, 
$ d s d^{-1}= s_{dA}= s_{\beta}\in \Wf$.

\end{proof}

There is an action of the group $\Gp$ on the semisimple building $\mathscr X$ recalled  in \cite[p. 104]{SS} that extends the action of $N_\Gp(\rm T)$
on the standard apartment. 
For $F$ a standard facet,
we denote by ${\EuScript P}_F^\dagger$ the stabilizer of $F$ in $\G$.  \begin{prop}

\phantomsection

\begin{itemize}
\item[i.] The Iwahori subgroup $\Iw$ acts transitively on the apartments of $\mathscr X$ containing $C$.
\item[ii.]  The stabilizer ${\EuScript P}_{x_0}^\dagger$  of $x_0$  acts transitively on the chambers of $\mathscr X$ containing $x_0$ in their closure.

\item[iii.] A $\G$-conjugate of $x_0$ in the closure of $C$ is a ${\EuScript P}_C^\dagger$-conjugate of $x_0$.

\end{itemize}
\label{prop:buildingstuff}

\end{prop}

\begin{proof} Point i is \cite[4.6.28]{BTII}. For  ii,  we first consider $C'$  a chamber of $\Ap$ containing $x_0$ in its closure.
The group $\W$ acting transitively on the chambers of $\Ap$, there is 
$d\in \Dd$ and $w_0\in \Wf$ such that $C'=w_0 d C$ and
$C$ contains $d^{-1} x_0$ in its closure. By \cite[Proposition 4.13 i.]{OS}, it implies that $d^{-1}C=C$ and therefore $C'=w_0 C$  or, when considering the action of $G$ on the building,   $C'= \hat w_0 C$ where $\hat w_0\in \K\cap N_G(\T)$ denotes a lift for $w_0$.
Now let $C''$ be a chamber of $\mathscr X$ containing $x_0$ in its closure.
By \cite[Corollaire 2.2.6]{BTI}, there is $k\in {\EuScript P}_{x_0}^\dagger$  such that $k C''$ is in $\Ap$.  Applying the previous observation,  $C''$ is a ${\EuScript P}_{x_0}^\dagger$-conjugate of $C$.
Lastly, let $gx_0$ (with $g\in \Gp$) be a conjugate of $x_0$ in the closure of $C$. By ii, the chamber $g^{-1}C$ is of the form $kC$  for $k\in {\EuScript P}_{x_0}^\dagger$ which implies that $gk\in {\EuScript P}_C^\dagger$ and
$gx_0$ is a ${\EuScript P}_C^\dagger$-conjugate of $x_0$.
 
\end{proof}

\begin{rema}\label{rema:conjug-stabi}
By \cite[Lemma 4.9]{OS},  ${\EuScript P}_C^\dagger$ is the disjoint union of
all $\Iw \hat \omega \Iw= \hat \omega \Iw$ for $\omega\in \Omega$. Therefore, a $\G$-conjugate of $x_0$ in the closure of $C$ is a ${\EuScript P}_C^\dagger\cap N_\G(\T)$-conjugate of $x_0$.

\end{rema}

\subsubsection{Weyl chambers\label{cones}}
The set of dominant coweights  $\X_*^+(\Tp)$ is the set of all $\lambda\in \X_*(\Tp)$ such that $\lp \lambda, \alpha\rp\geq 0$ for all $\alpha\in \Phi^+$.  It is called the dominant chamber.
Its opposite is the antidominant chamber.
A coweight $\lambda$ such that  $\lp \lambda, \alpha\rp> 0$ for all $\alpha\in \Phi^+$ is called strongly dominant. By \cite[Lemma 6.14]{BK}, 
strongly dominant elements do exist.
\medskip

We call a facet  $F$ of $\Ap$  standard if it is a facet of $C$   containing $x_0$ in its closure. 
Attached to a standard facet $F$ is the subset $\Phi_F$ of all   roots in $\Phi$ taking value zero on $F$ and the subgroup  $\Wf_F$ of $\Wf$ generated by the simple reflections stabilizing $F$. Let $\Phi_F^+:=\Phi^+\cap \Phi_F$ and
$\Phi_F^-:=\Phi^-\cap \Phi_F$. Define the following  Weyl chamber in $\X_*(\T)$  as in \cite[4.1.1]{invsatake}:
 $$\Cute^+(F)=\{\lambda\in  \X_*(\T)\textrm{  such that }\lp\lambda,\alpha\rp\geq 0\textrm{ for all }\alpha\in (\Phi^+-\Phi^+_F)\cup \Phi_F^-\}$$ and its opposite $\Cute^-(F)=-\Cute^+(F).$ 
 They are respectively the images of the dominant and antidominant chambers by the longest element $w_F$ in $\Wf_F$.
 
 \medskip
 By  Gordan's Lemma  (\cite[p.~7]{KKMS}), a Weyl chamber is  finitely generated as a semigroup.
 %Note that the facet $F$ here is in fact only  used to define the subset $\Pi_F$ of $\Pi$ of the simple roots taking values on $F$. Then $\Pi_F$ is a basis for $\Phi_F^+$. We could therefore identify  the Weyl chambers simply by the datum of  a sign ($+$ or $-$) and a subset of $\Pi$. 
 
 \subsubsection{\label{tilda}} We follow the notations of \cite[2.2.2, 2.2.3]{invsatake}.  
Recall that $\Tp^1$ is the pro-$p$ Sylow subgroup of $\T^0$.
We denote by $\tilde\W$ the quotient of $N_\Gp(\Tp)$   by $\Tp^1$  and obtain the exact sequence
$$0\rightarrow \Tp^0/\Tp^1 \rightarrow \tilde\W\rightarrow \W\rightarrow 0.$$ 
The group $\tilde\W$  parametrizes the double cosets of $\Gp$ modulo $\I$.  We fix a lift $\hat w\in N_\G(\T)$ for any $w\in \tilde \W$ and denote by $\tau_w$ the characteristic function of the double coset $\I \hat w \I$. The set of all $(\tau_w)_{w\in \tilde \W}$ is a $\Z$-basis for $\Hh_\Z$ which was defined  in the introduction to be  the convolution ring  of  $\Z$-valued functions with compact support in $\I\backslash \Gp
/\I$. For $g\in \Gp$, we will also use the notation $\tau_g$ for the characteristic function of the double coset $\I g \I$. \medskip

For $Y$ a subset of $\W$, we denote by $\tilde Y$ its preimage in $\tilde \W$. In particular, we have   the preimage $\tilde \X_*(\T)$  of $\X_*(\T)$. As well as those of $\X_*(\T)$, its elements will be denoted by $\lambda$ or $e^\lambda$ and called coweights. For $\alpha\in \Phi$, we inflate the function $\lp \,.\, , \alpha\rp$ defined on $\X_*(\T)$ to 
$\tilde \X_*(\T)$.   We still call \emph{dominant coweights} the elements in the preimage  $\tilde \X_*^+(\T)$ of $\X_*^+(\T)$. For $\sigma$ a sign and $F$ a standard facet, we consider the preimage of $\Cute^\sigma(F)$  in $\tilde \X_*(\T)$ and we still denote it by $\Cute^\sigma(F)$.

\medskip

The length function $\ell$ on $\W$ pulls back to a length function $\ell$ on $\tilde\W$  (\cite[Proposition 1]{Ann}).
For $u,v\in \tilde\W$ we write $u\leq v$ (resp. $u<v$) if their projections $\bar u$ and $\bar v$ in $\W$ satisfy $\bar u\leq \bar v$
(resp. $\bar u<\bar v$).
\subsubsection{\label{par:split}} We emphasize the following remark which will be important for the definition of the subring $\Zz^\c(\Hh_\Z)$ of the center of $\Hh_\Z$ in \ref{center}. 
\medskip

For $\lambda\in \X_*^+(\T)$,  the element  $\lambda(\varpi^{-1})\in N_\Gp(\T)$ is a lift for $e^\lambda$ seen in $\W$ 
by our convention \eqref{normalization}.
The map
\begin{equation}\label{splitting}\lambda\in \X_*(\T)\rightarrow [ \lambda(\varpi^{-1})\,{\rm mod} \,\T^1 ]\in \tilde \X_*(\T)\end{equation}
  is a $\Wf$-equivariant splitting for   the exact sequence of abelian groups
\begin{equation}\label{eq:split}0\longrightarrow \T^0/\T^1\longrightarrow \tilde\X_*(\T) \longrightarrow \X_*(\T)\longrightarrow 0.\end{equation}
We will identify $\X_*(\T)$ with its image in $\tilde \X_*(\T)$ \emph{via} \eqref{splitting}.  
Note that this identification depends on the choice of the uniformizer $\varpi$.

\begin{rema}\label{rema:semi}
We have the decomposition of $\tilde \W$ as a semidirect product $\tilde \W=\tilde\Wf\ltimes \X_*(\T)$ where  $\tilde \Wf$ denotes the preimage of $\Wf$ in $\tilde\W$.
\end{rema}
%The length function $\ell$ on $\W$ pulls back to a length function $\ell$ on $\tilde\W$. 
%Denote by $\leq $ the Bruhat order on $\W$. For $\tilde w, \tilde w'\in \tilde \W$, we write $\tilde w\leq \tilde w'$

\subsubsection{Pro-$p$ Hecke rings\label{defi:rings}}  
The product in   the generic pro-$p$ Iwahori-Hecke ring  $\Htg$  is  described in  \cite[Theorem 1]{Ann}. 
It is given by \emph{quadratic relations} and \emph{braid relations}. Stating the quadratic relations in $\Htg$  requires some more notations. We are only going to use them in $\Hh_k$ where they have a simpler form,  and we postpone their description to \ref{idempo}.
We recall here the braid relations:
\begin{equation} \label{braid} \textrm{$\t_{ww'}=\t_w \t_{w'}$ for $w, w'\in \tilde \W$ satisfying $\ell(ww')=\ell(w)+ \ell(w')$.}\end{equation}
The functions in $\Hh_\Z$ with support in the subgroup of $\Gp$ generated by all parahoric subgroups form a  subring $\Hh_\Z^{aff}$   called the affine subring.
It has $\Z$-basis the set of all $\tau_w$ for $w$   in the preimage  $\tilde\W_{aff}$ of $\W_{aff}$ in $\tilde\W$ (see for example \cite[4.5]{OS}).
It is generated by all $\tau_{s}$ for $s$ in the preimage $\tilde S_{aff}$ of $S_{aff}$ and all $\tau_t$ for $t\in \T^0/\T^1$.

\medskip

There is  an involutive automorphism  defined on $\Hh_\Z\otimes_\Z\Z[q^{\pm 1/2}]$ by (\cite[Corollary 2]{Ann}): \begin{equation}
\label{defi:invo}\upiota:\t_w\mapsto (-q)^{\ell(w)}\t_{w^{-1}}^{-1}.\end{equation}   and it actually yields an involution on $\Hh_\Z$.
 Inflating the character $\epsilon_C: \W\rightarrow \{\pm 1\}$ defined in \ref{orientation} to a character of $\tilde\W$,
we define the following $\Z$-linear involution $\upsilon_C$ of $\Hh_\Z$ by:
\begin{equation*}
    \upsilon_C(\tau_w) = \epsilon_C(w) \tau_w \qquad\text{for any $w \in \tilde{\W}$}.
\end{equation*}
It  is the identity on the affine subring $\Hh^{aff}_\Z$. We will consider the following $\Z$-linear involution on $\Hh_\Z$:\begin{equation}
\upiota_C=\upiota\circ \upsilon_C.\label{therightinvolution}
\end{equation}

\begin{rema} The involution   $\upiota$  fixes all $\tau_{w}$ for $w\in\tilde\W$ with length zero. 
The involution   $\upiota_C$  fixes all $\tau_{e^\lambda}$ for $\lambda\in\tilde\X_*(\T)$ with length zero. 
\label{fix0}
\end{rema}

\subsubsection{\label{idempo}}

Let $\R$ be a ring  with unit $1_\R$, 
containing an inverse for $(q1_\R-1)$ and a primitive $(q-1)^{\rm{th}}$ root of $1_\R$.
The group of  characters of $ \T^0/\T^1= \overline {\mathbf T}(\fq)$ with values in $\R^\times$ is isomorphic to the group
 of characters of $ \overline {\mathbf T}(\fq)$ with values in $\fq^\times$ which we denote by $ \hat{\overline {\mathbf T}}(\fq)$.
To $\xi\in  \hat{\overline {\mathbf T}}(\fq)$   we attach
 the idempotent element $\epsilon_\xi\in \Hh_\R$
 as in \cite{Ann} (definition recalled in \cite[2.4.3]{invsatake}). 
 For $t\in \T^0$  we have 
$\epsilon_{\xi} \t_{t}=\t_{t}\epsilon_{\xi} =\xi(t)\epsilon_{\xi}.$
 The idempotent elements $\epsilon_\xi$, ${\xi\in  \hat{\overline {\mathbf T}}(\fq)}$, are pairwise orthogonal
and their sum
 is the identity in $\Hh_\Z\otimes _\Z\R$.
 
 For $A\in \Pi_{aff}$, choose the lift $n_A\in \Gp$ for $s_A$ 
 defined after fixing an épinglage for $\Gp$ as in \cite[1.2]{Ann}.  We refer to \cite[2.2.5]{invsatake}
 for the definition of the associated subgroup $\T_A$ of $\T^0$ which identifies with a subgroup of $\T^0/\T^1$.

 \medskip
 
 For  $\xi\in  \hat{\overline {\mathbf T}}(\fq)$, we have in  $\Hh_\Z\otimes _\Z\R$:

\begin{equation}\label{Qgene'}\left\lbrace\begin{array}{l}
\textrm{ if $\xi$ is trivial on $\T_A$,   then    $\epsilon_{\xi}\t_{n_A}^2=\epsilon_{\xi}((q1_\R-1) \t_{n_A}+ q 1_\R)$}\cr
\textrm{ otherwise $\epsilon_{\xi}\t_{n_A}^2\in q\R^\times \epsilon_{\xi}$}.\cr\end{array}\right.\end{equation}

 The field $k$ is an example of ring $\R$
 as above.  In $\Hh_k$ we have:  

\begin{equation}\label{Q'}\left\lbrace\begin{array}{l}
\textrm{ if $\xi$ is trivial on $\T_A$,   then    $\epsilon_{\xi}\t_{n_A}^2=-\epsilon_{\xi} \t_{n_A}$}\cr
\textrm{ otherwise $\epsilon_{\xi}\t_{n_A}^2=0$}.\cr\end{array}\right.\end{equation}

\begin{rema} \label{rema:prodnul} 
In $\Hh_k$ we have $\tau_{n_A}\,\upiota(\tau_{n_A})=0$ for all $A\in S_{aff}$.
Furthermore, $\upiota(\tau_{n_A})+\tau_{n_A}$ lies in the subalgebra of $\Hh_k$ generated by all $\tau_t$, $t\in \T^0/\T^1$, or equivalently
by all $\epsilon_\xi$, $\xi\in \hat{\overline {\mathbf T}}(\fq)$. This can  be seen using  for example \cite[Remark 2.10]{invsatake}, which also implies the following: \\ 
if $\xi$ is trivial on $\T_A$, then $\upiota(\epsilon_\xi\tau_{n_A})=\epsilon_\xi\upiota (\tau_{n_A})=-\epsilon_\xi (\tau_{n_A}+1)$and\\
if $\xi$ is not trivial on $\T_A$ , then  $\upiota(\epsilon_\xi\tau_{n_A})=-\epsilon_\xi \tau_{n_A}$.

\end{rema}

\subsubsection{\label{weights} Parametrization of the weights}  
The functions in $\Hh_\Z$ with support in $\K$ form a  subring $\H_\Z$.
It has $\Z$-basis the set of all $\tau_w$ for $w\in  \tilde\Wf$. 
Denote by $\H_k$ the $k$-algebra $\H_\Z\otimes_\Z k$.
The  simple modules of $\H_k$ are one dimensional \cite[(2.11)]{Sawada0}.

An irreducible smooth $k$-representation $\rho$ of $\K$ will be called a weight. By  \cite[Corollary 7.5]{CL}   the weights are in one-to-one correspondence with the characters  of  $\H_k$ \emph{via} $\rho\mapsto \rho^\I$. 
To a  character $\chi: \H_k\rightarrow k$ is  attached    the morphism $\bar\chi: \T^0/\T^1\rightarrow k^\times$  such that $\bar\chi(t)=\chi(\t_t)$ for all $t\in \T^0/\T^1$ and
the  set $\Pi_{\bar\chi}$ of   all simple roots $\alpha\in \Pi$
such that $\bar\chi$ is trivial on $\T_\alpha$.
We then have $\chi(\t_{\tilde s_\alpha})=0$ for all $\alpha\in \Pi-\Pi_{\bar\chi}$,  where $\tilde s_\alpha\in \tilde\W$ is any lift for $s_\alpha\in \W$. We denote by $\Pi_{\chi}$ the subset of all
 $\alpha\in\Pi_{\bar\chi}$ such that $\chi(\t_{\tilde s_\alpha})=0$. 
The character $\chi $ is determined by the data of $\bar\chi$ and $\Pi_\chi$ (see also \cite[3.4]{invsatake}).

\begin{rema}\label{thefacet}
Choosing a standard facet $F$ is equivalent to choosing the subset $\Pi_F$ of $\Pi$ of the simple roots taking value zero on $F$.
 The standard facet  corresponding to $\Pi_\chi$ in  the previous discussion will be
  denoted by $F_\chi$.
\end{rema}

\section{\label{firstpart}On the center  of the pro-$p$ Iwahori-Hecke algebra in characteristic $p$}

\subsection{\label{modified}Commutative subrings of the  pro-$p$ Iwahori-Hecke ring}Let $\sigma$ be a sign and $F$ a standard facet.

\subsubsection{} As in \cite[4.1.1]{invsatake}, we introduce
the multiplicative injective map $$\Theta_F^\sigma:\tilde\X_*(\T)\longrightarrow \Htg\otimes_\Z \Z[q^{\pm 1/2}]$$    and the elements
$\Bb_F^\sigma(\lambda):=q^{\ell(e^\lambda)/2} \Theta_F^\sigma(\lambda)$ for all $\lambda\in \tilde\X_*(\T).$ Recall that $\Bb_F^\sigma(\lambda)=  \t_{e^\lambda}$ if $\lambda\in \Cute^\sigma(F)$.
%(and this remark is sufficient to retrieve the  formula for $\Bb_F^\sigma(\lambda)$ for  a general $\lambda$). 
   The map $\Bb_F^\sigma$  does not respect the product in general, but it is multiplicative when restricted to any Weyl chamber (see \cite[Remark 4.3]{invsatake}).
For any coweight  $\lambda\in \tilde\X_*(\T)$, the element 
$\Bb_F^\sigma(\lambda)$ lies in $\Hh_\Z$ (see Lemma \ref{lemma:fundam} below).
Furthermore combining  Lemmas \ref{calcu1}ii.,   \ref{lemma:fundam} and \cite[Lemma  4.4]{invsatake}:
\begin{equation} \label{involution} \upiota_C(\Bb_F^+(\lambda))=\Bb_F^-(\lambda).
\end{equation}
Extend $\Theta_F^\sigma$ linearly  to an injective morphism of $\Z[q^{\pm 1/2}]$-algebras $\Z[q^{\pm 1/2}][\tilde\X_*(\T)]\longrightarrow \Htg\otimes_\Z \Z[q^{\pm 1/2}]$. 
We consider the commutative subring 
$ \Aa^\sigma_F:= \Htg\cap \Im(\Theta_F^\sigma).$
By \cite[Prop. 4.5]{invsatake}, it is a free $\Z$-module with basis the set of all $\Bb_F^\sigma(\lambda)$ for $\lambda \in \tilde \X_*(\T)$. Since the Weyl chambers (in $\tilde \X_*(\T)$) are finitely generated semigroups,  $ \Aa^\sigma_F$ is finitely generated as a ring.

\begin{rema}\label{compa}Note that  $\Bb_{C}^+=\Bb_{x_0}^-$ (resp.  $\Bb_{C}^-=\Bb_{x_0}^+$) coincides with the integral Bernstein map $E^+$ (resp. $E$) introduced in \cite{Ann} and $\Aa_ C^+$ (resp.  $ \Aa_ C^-$) with the commutative ring  denoted by $\cal A^{+,(1)}$ (resp. 
$\cal A^{(1)}$) in \cite[Theorem 2]{Ann}.
 
 \end{rema}
 
\noindent Identify $\X_*(\T)$ with its image in $\tilde \X_*(\T)$ \emph{via} \eqref{splitting}.
 We denote by $(\Aa^\sigma_F)^\c$ the intersection 
\begin{equation*}(\Aa^{\sigma}_F)^\c :=\Htg \cap \Theta^\sigma_F(\Z[\X_*(\T)])\subseteq \Aa^\sigma_F. \end{equation*}
A $\Z$-basis for $(\Aa^\sigma_F)^\c$ is given by all $\Bb_F^\sigma(\lambda)$ for $\lambda\in \X_*(\T)$. It is finitely generated as a ring.

\begin{prop} \label{tensor}
The commutative $\Z$-algebra $\Aa_F^\sigma$ is isomorphic to the tensor product of  the $\Z$-algebras $\Z[\T^0/\T^1]$ and $(\Aa_F^\sigma)^\c$. In particular,  $(\Aa_F^\sigma)^\c$ is a direct summand of 
 $\Aa_F^\sigma$ as a $\Z$-module.
\end{prop}

\begin{proof}  
Since  the exact sequence \eqref{eq:split} splits,  $\Aa_F^\sigma$ is  a free
$(\Aa_F^\sigma)^\c$-module with basis  the set of all $\t_t$ for  $t\in \T^0/\T^1$. Recall indeed that $\Bb_F^\sigma(\lambda+t)=\Bb_F^\sigma(\lambda)\t_t= \t_t\Bb_F^\sigma(\lambda)$ for all $\lambda\in \tilde \X_*(\T)$ and
$t\in \T^0/\T^1$. 
\end{proof}

\subsubsection{\label{fund}} The following is  a direct consequence of the lemma proved in \cite[\S5]{Haines} and adapted to the pro-$p$ Iwahori-Hecke algebra  in \cite[Lemma 13]{Ann} (see also \cite[1.2 and 1.5]{VigGene}).

\begin{lemma}\label{lemma:fundam} Let $F$ be a standard facet and $\sigma$ a sign.
For any $\lambda\in \tilde \X_*(\T)$, we have
$$\Bb_F^\sigma(\lambda)= \tau_{e^\lambda}+\sum_{w<e^\lambda} a_w \tau_w$$ where  $(a_w)_w $
is a family of elements in $\Z$ (depending on $\sigma$, $F$ and $\lambda$)
indexed by the set of $w\in \tilde \W$ such that $w<e^\lambda$. For those  $w$, we have in particular $\ell(w)<\ell(e^\lambda)$.

\end{lemma}

\subsubsection{\label{not:irreducibleroot}}  In this paragraph, we suppose that the root system
of  $\Gp$ is  irreducible. It implies in particular that  there is a unique element in $\Pi_m$. It can be written $-\alpha_0$ where $\alpha_0\in \Phi^+$ is the highest root: we have
$\beta\preceq \alpha_0$ for all $\beta\in \Phi$ (\cite[VI. n$^{\circ}1.8$]{Bki-LA}).  For any standard facet $F\neq x_0$, we have  $\alpha_0\not\in \Phi_F$.
Denote by $s_0\in S_{aff}$ the simple reflection associated to $(-\alpha_0, 1)\in \Pi_{aff}$ and $n_{0}:=n_{(-\alpha_0,1)}\in \Gp$ the lift for $s_0$ as chosen in \ref{idempo}.

\begin{lemma} Suppose that $F\neq x_0$ and let $\lambda\in \tilde\X_*^+(\T)$ such that $\ell(e^\lambda)\neq 0$. We have $$\Bb_F^+(\lambda)\in \tau_{n_0}\Hh_\Z\,.$$
\label{calcul-n0}
\end{lemma}
\begin{proof} It suffices to check the claim for $\lambda\in \X_*^+(\T)$.
Let $\mu, \nu\in \X_*(\T)$ such that $\lambda=\mu-\nu$ and   $w_F\mu, w_F\nu\in  \X_*^+(\T)$  where we recall that  $w_F$  denotes the longest element in $\Wf_F$. Note that $w_F\alpha_0\in \Phi^+$ because $F\neq x_0$. 
Furthermore,  $\lp \lambda, \alpha_0\rp\geq 1$ because there is $\beta\in \Pi$ such that $\lp \lambda, \beta\rp\geq 1$ and $ \beta\preceq \alpha_0$.

We have $e^\nu(-\alpha_0, 1)= (-\alpha_0, 1+\lp \nu, \alpha_0\rp)=(-\alpha_0, 1+\lp w_F\nu, w_F\alpha_0\rp)\in \Phi_{aff}^+$. Therefore $ \ell(e^\nu n_0)=\ell(e^\nu)+1$ and $\tau_{e^\nu}\tau_{n_0}=\tau_{e^\nu n_0}$ in $\Hh_\Z$. On the other hand $e^{-\lambda}(-\alpha_0, 1)= (-\alpha_0, 1-\lp \lambda, \alpha_0\rp)\in \Phi_{aff}^-$ and therefore $ \ell(n_0e^\lambda )=\ell(e^\lambda)-1$.

We perform the computations in $\Htg\otimes_\Z \Z[q^{\pm 1/2}]$ where by definition, $\Bb_F^+(\lambda)=q^{\frac{\ell(e^\lambda)+\ell(e^\nu)-\ell(e^\mu)}{2}}  \tau_{e^\nu}^{-1}\tau_{e^\mu}$.
By the previous remarks
\begin{align*}\Bb_F^+(\lambda)&=\tau_{n_0}q^{\frac{\ell(n_0e^\lambda)+\ell(e^\nu n_0)-\ell(e^\mu)}{2}}  \tau_{e^\nu n_0}^{-1}\tau_{e^\mu} \end{align*} which, by the  lemma evoked in \ref{fund}, lies in $\tau_{n_0}\Hh_\Z$.

\end{proof}

\medskip
\subsection{\label{center}On the center  of  the pro-$p$ Iwahori-Hecke  ring} 
\subsubsection{}  The ring  $\Htg$ is finitely generated as a module over its center $\Zz(\Htg)=(\Aa_C^+)^\Wf$  and  the latter has $\Z$-basis the set of all 
\begin{equation}\label{Z1}
\sum_{\lambda'\in \OO} \Bb_{C}^+(\lambda')\end{equation} where $\OO$ ranges over the $\Wf$-orbits in 
$\tilde\X_*(\T)$.   Moreover, $\Zz(\Htg)$ is a finitely generated $\Z$-algebra.  Those results are  proved in \cite[Theorem 4]{Ann} (note that the hypothesis of irreducibility of the root system of $\G$  in \cite{Ann} is not necessary for the statements about the center). One can also find a proof in \cite{NS}.

\subsubsection{} We denote by $\Zz^\c(\Htg)$ the intersection of $(\Aa^+_C)^\c$ with $\Zz(\Htg)$. 
We have $\Zz^\c(\Htg)=((\Aa^+_C)^\c)^\Wf$. 
It has $\Z$-basis the set of all 
 \begin{equation}\label{zl} z_\lambda:=\sum_{\lambda'\in\OO(\lambda)} \Bb_C^+(\lambda')  \:\textrm{ for $\lambda\in \X_*^+(\T)$}\end{equation}
where  we denote by $\OO(\lambda)$ the $\Wf$-orbit of $\lambda$.

\begin{prop} \begin{enumerate}
\item[i.]  The left (resp. right) $(\Aa_C^+)^\c$-module  $\Htg$ is finitely generated.
\item[ii.]  As a $\Zz^\c(\Htg)$-module, $\Htg$ is finitely generated.
\item[iii.]  $\Zz^\c(\Htg)$ is a finitely generated $\Z$-algebra.
\item[iv.]  As $\Z$-modules,  $\Zz(\Htg)$, $\Aa_C^+$, $\Zz^\c(\Htg)$ and  $(\Aa_C^+)^\c$ are direct summands of $\Hh_\Z$.
\end{enumerate}
\label{prop:finite}

\end{prop}
\begin{proof}

Using  Proposition \ref{tensor} and \cite[Theorems 3 and 4]{Ann} which state that $\Htg$ is finitely generated over $\Aa_C^+$ (see Remark \ref{compa}),  we  see  that  $\Htg$ is finitely generated over $(\Aa_C^+)^\c$. Statements ii. and iii. follow from \cite[\S1 n. 9  Thm 2]{Bki-AC}
because 
$\Zz^\c(\Htg)$ is the  ring  of $\Wf$-invariants of  $(\Aa_C^+)^\c$ and $\Z$ is noetherian.
For iv., we first remark that the $\Z$-module $\Zz(\Htg)$ (resp. $\Zz^\c(\Htg)$)  is a direct summand of $\Aa_C^+$ (resp. $(\Aa_C^+)^\c$) since $\Zz(\Htg)=(\Aa^+_C)^\Wf$ (resp. $\Zz^\c(\Htg)=((\Aa^+_C)^\c)^\Wf$).
The $\Z$-module  $(\Aa_C^+)^\c$ is a direct summand of $\Aa_C^+$ by Proposition \ref{tensor}. 
It remains to show that  $\Aa_C^+$ is a direct summand of $\Hh_\Z$ which can be done by considering the integral Bernstein basis for the whole Hecke ring $\Hh_\Z$  introduced in \cite{Ann}. We recall it later in \ref{intbasis} and finish the proof of iv. in  Remark \ref{directsummand}.
\end{proof} 

%\begin{coro}\label{extendZ} 
%Let $R$ be a ring.
%A morphisms of rings $\Zz^\c(\Htg)\rightarrow R$  can be extended  to a morphism of rings $\Zz(\Htg)\rightarrow R$.

%\end{coro}

%\begin{proof} Let $\zeta: \Zz^\c(\Htg)\rightarrow \R$  a character.
%The $\Htg$-module $\Htg\otimes_{\Zz^\c(\Htg)}\zeta$ is finite dimensional and therefore contains a  nonzero submodule with  a central character: the latter coincides with $\zeta$ on $\Zz^\c(\Htg)$.

%\end{proof}
\subsubsection{\label{identif}} 
Given a ring $\R$ with unit $1_\R$, we denote by $\Hh_\R$ the $\R$-algebra $\Htg\otimes_\Z \R$: we identify $q$ with its image  in $\R$. By Proposition \ref{prop:finite}iv., the  $\R$-algebra
$\Zz(\Hh_\Z)\otimes_\Z \R$ (resp.  $\Aa_C^+\otimes _\Z \R$, $(\Aa_C^+)^\c\otimes _\Z \R$ and $\Zz^\c(\Hh_\Z)\otimes _\Z \R$) identifies with a subalgebra of $\Hh_\R$ which we  denote by  $\Zz(\Hh_\R)$ (resp. $(\Aa_C^+)_R$, $(\Aa_C^+)^\c_R$ and $\Zz^\c(\Hh_\R)$). By the work of \cite{NS},  $\Zz(\Hh_\R)$ is not only contained in but equal to the center of 
$\Hh_\R$.\\

\begin{rema} Proposition  \ref{prop:finite} remains valid with $x_0$ instead of $C$ (use the involution $\upiota_C$ and \eqref{involution}). We introduce the subalgebras $(\Aa_{x_0}^+)_R$ and $(\Aa_{x_0}^+)^\c_R$ of $\Hh_\R$
with the obvious definitions.
\end{rema}

For $\lambda\in \tilde\X_*(\T)$ (resp. $w\in \tilde \W$), we  still denote by $\Bb_F^\sigma(\lambda)$ (resp.    $\t_w$)  its  natural image $\Bb_F^\sigma(\lambda)\otimes 1$ (resp.    $\t_w\otimes 1$)  in $\Hh_\R$. An $\R$-basis for $\Zz^\c(\Hh_\R)$ is given by the set of all $z_\lambda$ for $\lambda\in \X_*^+(\T)$, where again we identify the element $z_\lambda$ with its image in $\Hh_\R$.

\medskip
From Proposition \ref{prop:finite} we deduce:
\begin{prop}\label{extendZ} 
Let $\R$ be a field.
A morphism of $\R$-algebras  $\Zz^\c(\Hh_\R)\rightarrow \R$  can be extended  to a morphism of $\R$-algebras $\Zz(\Hh_\R)\rightarrow \R$.

\end{prop}

\subsubsection{} In the process of constructing $\Zz^\c(\Hh_\Z)$, we first fixed  a hyperspecial vertex  $x_0$ of $C$ and then
an apartment $\Ap$ containing $C$.

\begin{prop} \label{prop:conjugacy} The ring 
$\Zz^\c(\Hh_\Z)$ is not affected by 
\begin{itemize}
\item the choice of another apartment $\Ap'$ containing $C$.
\item the choice of another vertex  $x_0'$ of $C$  provided it is $\Gp$-conjugate to $x_0$.
\end{itemize}

\end{prop}
\begin{proof} 
Let $g$ in the stabilizer ${\EuScript P}_{C}^\dagger$ of $C$ in $\G$.
Let $\T':=g \T g^{-1}$ and $x_0'= g x_0 g^{-1}$. The apartment $\Ap'$ corresponding to $\T'$ contains $C$ and $x_0'$ is a hyperspecial vertex of $C$. Starting from $\T'$ and $x_0'$ we  proceed to the construction of the corresponding commutative subring $\Zz^\circ (\Hh_\Z)'$ of the center  of $\Hh_\Z$. 
Since $g\in {\EuScript P}_{C}^\dagger$,  we have
 $\I g\I=\I\hat \omega \I=\I \hat \omega$ for some   $\omega \in \tilde\Omega$. Since this element $\omega$ has length zero,   for $\lambda\in \X_*(\T)$
the characteristic function  of $\I g\lambda (\varpi)g^{-1}\I$ is equal to the product $\tau_{g}\tau_{\lambda(\varpi)}\tau_g^{-1}$. 
Therefore,  the restriction to $\X_*(\T)$ of the new map $(\Bb_C^+)'$ corresponding to the choice of $x_0'$ and $
\T'$   
is defined by
$$ \X_*(\T')\longrightarrow \Hh_\Z, \: \lambda\mapsto \tau_{g} \Bb_C^+( g^{-1}\lambda g)\tau_{g}^{-1}.$$
The element $z'_\lambda\in \Zz^\circ (\Hh_\Z)'$ corresponding to the choice of $\lambda\in \X^+_*(\T')= g  \X^+_*(\T) g^{-1}$ is therefore $\tau_{g}z_{g^{-1}\lambda g}\tau_g^{-1}= z_{\lambda}$.
We have proved that $\Zz^\circ (\Hh_\Z)'=\Zz^\circ (\Hh_\Z)$.

\medskip

By Proposition \ref{prop:buildingstuff}i
 and  Remark \ref{rema:conjug-stabi}
 \begin{itemize}
 \item changing $\Ap$ into another apartment $\Ap'$ containing $C$ and 
 \item changing
 $x_0$ into another vertex $x_0'$ of $C$ which is $\Gp$-conjugate to $x_0$ 
 \end{itemize}
 can be made independently of each other by conjugating by an element of 
  $\Iw$ and  of  ${\EuScript P}_C^\dagger\cap N_\G(\T)$ respectively.
We have checked that these changes do not affect $\Zz^\circ (\Hh_\Z)$.

\end{proof}
If $\mathbf G$ is of adjoint type or $\mathbf G={\rm GL}_n$, then all hyperspecial vertices are conjugate (\cite[2.5]{Tit}): 
\begin{coro}\label{coro:adjoint}
If $\mathbf G$ is  of adjoint type or $\mathbf G={\rm GL}_n$, then $\Zz^\c(\Hh_\Z)$ depends only on the choice of the uniformizer $\varpi$.
\end{coro}

\subsection{\label{defi:spheri}An affine semigroup algebra in the center of the pro-$p$ Iwahori-Hecke algebra in characteristic $p$} 
%In this paragraph we work with the $k$-algebra $\Hh_k$ and we give  the structure of $\Zz^\c(\Hh_k)$.

We will use the following observation several times in this paragraph.  Let $F$ be a standard facet and $\sigma$ a sign.
For $\mu_1,\mu_2\in \X_*(\T)$, we have in $\Hh_k$:
\begin{equation}  \label{fact:annu}\Bb_F^\sigma(\mu_1) \Bb_F^\sigma(\mu_2)= \left\lbrace\begin{array}{ll} \Bb_F^\sigma(\mu_1+\mu_2)&\emph{if $\mu_1$ and $\mu_2$ lie in a common Weyl chamber }\cr 0&\emph{otherwise}. \end{array}\right.\end{equation}

 In $\Htg\otimes_\Z\Z[q^{\pm 1/2}]$ we have indeed
 $\Bb_F^\sigma(\mu_1) \Bb_F^\sigma(\mu_2)=
q^{(\ell(e^{\mu_1})+\ell(e^{\mu_2})-\ell(e^{\mu_1+\mu_2}))/2} \Bb_F^\sigma(\mu_1+\mu_2)$. 
If $\mu_1$ and $\mu_2$ lie in a common Weyl chamber, then
 $\ell(e^{\mu_1})+\ell(e^{\mu_2})-\ell(e^{\mu_1+\mu_2})$ is zero; otherwise,   there is $\alpha\in \Pi$ satisfying $\lp\mu_1, \alpha\rp \lp\mu_2, \alpha\rp<0$ which implies that this quantity  is $\geq 2.$ 
This gives the required equality in $\Hh_k$. 
%$\ell(e^\mu_1)+\ell(e^{\mu_1'})-\ell(e^{\mu_1+\mu_1'})=\sum_{\beta\in \Phi^+}\vert\lp\mu_1,\beta\rp\vert+ \vert\lp\mu_1',\beta\rp\vert- \vert\lp\mu_1+\mu_1',\beta\rp\vert\geq \lp\mu_1,\alpha\rp+\vert \lp\mu_1',\alpha\rp\vert- \vert\lp\mu_1+\mu_1',\alpha\rp\vert\geq 2$

\subsubsection{The structure of $\Zz^\c(\Hh_k)$}
\begin{prop} \label{prop:Zmap}
The map \begin{equation}\label{Zmap} \begin{array}{ccc} k[\X_*^+(\T)]&\longrightarrow &  \Zz^\c(\Hh_k)\cr 
\lambda&\longmapsto &z_\lambda 
\end{array}\end{equation} is an isomorphism of $k$-algebras.
\end{prop}

\begin{proof} 
We already know that \eqref{Zmap}  maps a $k$-basis for $k[\X_*^+(\T)]$ onto a $k$-basis for $\Zz^\c(\Hh_k)$. We have to check that it respects the product. Let $\lambda_1, \lambda_2\in \X_*^+(\T)$ with respective $\Wf$-orbits  $ \OO(\lambda_1)$ and $\OO(\lambda_2)$. We consider the product
 $$z_{\lambda_1} z_{\lambda_2}= \sum_{\mu_1\in \OO(\lambda_1), \,\mu_2\in \OO(\lambda_2) }  \Bb_F^\sigma(\mu_1) \Bb_F^\sigma(\mu_2)  \in \Hh_k.$$
A Weyl chamber  in $\X_*(\T)$ is a $\Wf$-conjugate of $\X_*^+(\T)$.
Given a Weyl chamber and a coweight (in  $\X_*(\T)$), there is a unique $\Wf$-conjugate of the coweight in the chosen Weyl chamber. 
The map $(\mu_1,\mu_2)\mapsto \mu_1+\mu_2$ yields a bijection between the set of all  $(\mu_1, \mu_2)\in \OO(\lambda_1)\times  \OO(\lambda_2)$ such that $\mu_1$ and $\mu_2$ lie in the same Weyl chamber  and the $\Wf$-orbit $\OO(\lambda_1+\lambda_2)$ of $\lambda_1+\lambda_2$: it is indeed surjective and one checks that the two sets in question have the same size because,  $\lambda_1$ and $\lambda_2$ being both dominant,
the stabilizer  in $\Wf$  of $\lambda_1+\lambda_2$ is  the intersection 
 of  the stabilizers    of $\lambda_1$ and  of $\lambda_2$. 
Together with  \eqref{fact:annu}, this proves that $z_{\lambda_1+\lambda_2}=z_{\lambda_1}z_{\lambda_2}$.
\end{proof}

For a different proof of this proposition, see  the remark after Theorem \ref{coro:diag}.

%Chambres vectorielles contenant $\lambda_1$: les $Stab(\lambda_1)$ conjuguées de $\X^+$. Eléments dans l'orbite de $\lambda_2$ appartenant à une même chambre que $\lambda_1$: 
%$\mu_2\in \OO(\lambda_2)$ such that there is $w\in Stab(\lambda_1)$ such that $\mu_2\in w\X^+$ that is $\mu_2= w\lambda_2$. Nombre de ces éléments: $Card(Stab(\lambda_1))/Card(Stab(\lambda_1)\cap Stab(\lambda_1))$.
%A $u$ fixé, éléments dans l'orbite de $\lambda_2$ appartenant à une même chambre que $\mu_1= u\lambda_1$: 
%les $u$ translatés des précédents.
%Nombre de paires $(\mu_1, \mu_2)$ dans la même chambre de Weyl:
%$Card(\Wf)/Card(Stab(\lambda_1))\times Card(Stab(\lambda_1))/Card(Stab(\lambda_1)\cap Stab(\lambda_1))$.

\subsubsection{}Since $\X_*(\T)$ is a free abelian group (of rank $dim(\T)$), the $k$-algebra  $k[\X_*(\T)]$ is isomorphic to an algebra of Laurent polynomials and has a trivial nilradical. By Gordan's Lemma, $\X_*^+(\T)$ is finitely generated as a semigroup. So $k[\X_*^+(\T)]$ is a finitely generated $k$-algebra and its Jacobson radical  coincides with its nilradical.
% (see for example [Milne, A primer of commutative algebra, Prop 11.8]).
The Jacobson radical  of  $\Zz^\c(\Hh_k)$ is therefore trivial.

\begin{prop}
The Jacobson radical  of $\Zz(\Hh_k)$ is trivial.

\end{prop}
\begin{proof} Since $\Zz(\Hh_k)$ is a finitely generated $k$-algebra contained in 
$(\Aa_C^+)_ k$, 
it is enough to prove that the nilradical of $(\Aa_C^+)_k$ is trivial. Using 
the notations of \ref{idempo},  it is enough to prove that, for any $\xi\in  \hat{\overline {\mathbf T}}(\fq)$,  the  nilradical of the $k$-algebra
 $\epsilon_\xi(\Aa_C^+)_ k$ with unit $\epsilon_\xi$ is trivial. 
By Proposition \ref{tensor}, 
the latter
algebra is isomorphic to  $(\Aa_C^+)^\c_ k$.
 It is therefore enough to prove that the nilradical of 
 $(\Aa_C^+)^\c_ k$ is trivial.

By definition 
(see the convention in \ref{identif}),  
the image of the $k$-linear injective map
$$\Bb_C^+:\k[ \X_*(\T)]\longrightarrow  \Hh_k$$  coincides with  $(\Aa_C^+)^\c_ k$.
\medskip

\begin{fact} Let $\lambda_0\in \X_*^+(\T)$ be a strongly dominant coweight.  
The  ideal  of  $(\Aa_C^+)^\c_ k$ generated by $\Bb_C^+(\lambda_0)$ does not contain any nontrivial nilpotent element.\label{ideal-nil}
\end{fact}

An element $a\in (\Aa_C^+)^\c_ k$ is a $k$-linear combination of elements $\Bb_C^+(\lambda)$ for $\lambda\in \X_*(\T)$ 
and we say that $\lambda\in \X_*(\T)$ is in the support of  $a$ if the   coefficient of $\Bb_C^+(\lambda)$  is nonzero.
Suppose that $a$ is nilpotent and nontrivial.  After conjugating by an element of $\Wf$, we can suppose that
there is an element of $\X_*^+(\T)$  in the support of $a$.
Then let $\lambda_0\in \X_*^+(\T)$ be  strongly dominant.
 The element $a\Bb_C^+({\lambda_0})$ is  nilpotent  and by  \eqref{fact:annu} it is nontrivial. By Fact \ref{ideal-nil}, we have a  contradiction.\\\\
\emph{Proof of the fact:}  The restriction  of $\Bb_C^+$  to $k[ \X^+_*(\T)]$ induces an isomorphism of $k$-algebras $k[ \X^+_*(\T)]\cong  \Bb_C^+(k[ \X^+_*(\T)])$. 
By \eqref{fact:annu}, the ideal  $\mathfrak A$ of  $(\Aa_C^+)^\c$ generated by $\Bb_C^+(\lambda_0)$
coincides with  the ideal  of $\Bb_C^+(k[ \X^+_*(\T)])$ generated by $\Bb_C^+(\lambda_0)$.
 Since the $k$-algebra $k[ \X^+_*(\T)]$  does not contain any  nontrivial nilpotent element, neither does $\mathfrak A$.

\end{proof}

Since $k$ is algebraically closed, we have:

\begin{coro}

Let $z\in\Zz(\Hh_k)$. If  $\zeta(z)=0$ for all characters $\zeta:\Zz(\Hh_k)\rightarrow k$, then $z=0$.
\label{coro:char}

\end{coro}

%\begin{proof}
%Recall that $k$ is algebraically closed, so  the maximal ideals of $\Zz(\Hh_k)$ are the kernels of the characters.
%\end{proof}

\subsubsection{The center of the Iwahori-Hecke algebra in characteristic $p$.}Let $\R$ be a ring  
containing an inverse for $(q1_\R-1)$ and a primitive $(q-1)^{\rm{th}}$ root of $1_\R$. We can apply the observations of \ref{idempo} and  consider the algebra 
$$\Hh_\R(\xi):=\epsilon_\xi \Hh_\R\epsilon_\xi.$$ It can be seen as the algebra $\cal H(\Gp, \Iw, \xi^{-1})$ of $\Gp$-endomorphisms of the representation
$\epsilon_\xi \ind_{\I}^\Gp {\mathbf 1}_\R$ which is isomorphic
to the compact induction $\ind_\Iw^\Gp \xi^{-1}$ of $\xi^{-1}$  seen as a $\R$-character of $\Iw$ trivial on $\I$: denote by $1_{\Iw, \xi^{-1}}\in \ind_\Iw^\Gp \xi^{-1}$ the unique function with support in $\Iw$ and value $1_\R$ at $1_\Gp$, then the map
\begin{equation}\label{isoiwa} \Hh_\R ({\xi})\rightarrow \cal H(\Gp, \Iw, \xi^{-1}), \: h\mapsto [1_{\Iw, \xi^{-1}}\mapsto 1_{\Iw, \xi^{-1}}h]
\end{equation} gives the identification. In particular, when $\xi=\mathbf 1$ is the trivial character, then the algebra $\Hh_\R(\mathbf 1)$ identifies with the usual   Iwahori-Hecke algebra 
$\HhI_\R=\R[\Iw\backslash \Gp/\Iw]$
with coefficients in $\R$.
%\begin{rema} \label{facet2}To a character $\xi$ is attached the set $\Pi_\xi$ of all $\alpha\in\Pi$ such that $\xi$ is fixed by $s_\alpha$ and the corresponding standard facet denoted by $F_\xi$ (see Remark \ref{thefacet}). For example, $F_{\mathbf 1}=x_0$, and if $\xi$ has  trivial  stabilizer  in $\Wf$,  then $F_\xi=C$ (but these  last two conditions are not equivalent).
%\end{rema}

\begin{rema} \label{centers} Let $\xi\in  \hat{\overline {\mathbf T}}(\fq)$. We have inclusions
$$\epsilon_{\xi }\Zz^\c(\Hh_\R)\subseteq  \epsilon_{\xi }\Zz(\Hh_\R)\subseteq \Zz(\Hh_\R(\xi))$$
where the latter  space  is the center of $\Hh_\R(\xi)$. The inclusion $\epsilon_{\xi }\Zz^\c(\Hh_\R)\subseteq \Zz(\Hh_\R(\epsilon_\xi))$ is strict in  general.  For example if
 $\G={\rm GL}_2(\Corps)$, $\R=k$, and $\xi$ is not fixed by the non trivial element of $\Wf$, then $\Hh_k(\xi)$ is commutative  with a $k$-basis indexed by the elements in $\X_*(\T)$ and contains zero divisors (\cite[Proposition 13]{BL}) while the  $k$-algebra $\epsilon _\xi \Zz^\c(\Hh_k)$  is isomorphic to $k[\X_*^+(\T)]$.

\medskip
If $\xi= \1$ however,  these inclusions are equalities: 
one  easily checks by direct comparison  of the basis elements \eqref{Z1} and \eqref{zl} that the first inclusion is an equality.  The second one comes from the fact  that $\epsilon_\1$ is a central idempotent in $\Hh_\R$.
 In particular we have:

\end{rema}

\begin{theorem}\label{centerIwa}
The center of the Iwahori-Hecke $k$-algebra $k[\Iw\backslash \Gp/\Iw]$   is isomorphic to $k[\X_*^+(\T)]$.
\end{theorem}

\begin{proof} The map  \begin{equation*}\begin{array}{ccc} k[\X_*^+(\T)]&\longrightarrow & \epsilon_{\mathbf 1} \Zz(\Hh_k)\cr 
\lambda&\longmapsto & \epsilon_{\mathbf 1} z_\lambda 
\end{array}\end{equation*} is surjective by the previous discussion.
It is easily checked to be injective using Lemma \ref{lemma:fundam} (compare with \cite[(1.6.5)]{VigGene}).
\end{proof}

\section{\label{BernCent}The central Bernstein functions in the pro-$p$ Iwahori-Hecke ring} 
 Let $\OO$ be a $\Wf$-orbit in 
$\tilde\X_*(\T)$.   We call the central element of $\Hh_\Z$
\begin{equation}\tag{\ref{Z1}}
z_\OO:=\sum_{\lambda'\in \OO} \Bb_{C}^+(\lambda')\end{equation} 
 the associated central Bernstein function.

\subsection{The support of the central Bernstein functions}
For $h\in \Hh_\Z$, the set of all $ w\in \tilde \W$ such that
$h(\hat {w})\neq 0$ is called the  \emph{support} of $h$.
For  $\OO$  a $\Wf$-orbit in 
$\tilde\X_*(\T)$ we denote by $\ell_\OO$ the common length of all the coweights in $\OO$.

\begin{lemma} Let $\OO$ be a $\Wf$-orbit in 
$\tilde\X_*(\T)$.
The support of $z_\OO$ (resp. $\upiota_C(z_\OO)$)
contains the set of  all  $e^\mu$ for  $\mu\in \OO$: more precisely, the coefficient of $\tau_{e^\mu}$ in the decomposition of $z_\OO$ (resp. $\upiota_C(z_\OO)$) is equal to $1$.  Any other element in the support of
$z_\OO$ (resp. $\upiota_C(z_\OO)$)  has length $<\ell_\OO$.

\label{decomZ}

\end{lemma}

\begin{proof}  This is a consequence of Lemma \ref{lemma:fundam} (and of \eqref{involution}).\end{proof}
\begin{prop} The involution $\upiota_C$ fixes the elements in the center $\Zz(\Hh_\Z)$ of $\Hh_\Z$.\label{prop:fix} \\In particular, for $\OO$ a $\Wf$-orbit in $\tilde\X_*(\T)$,
 the element
$\sum_{\lambda'\in \OO} \Bb_{C}^\sigma(\lambda')\in \Hh_\Z$ does not depend on the
sign $\sigma$.

\end{prop}

\begin{proof} 
We prove that $\upiota_C$ fixes $z_\OO$ by induction on $\ell_\OO$.\\

If $\ell_\OO=0$, then conclude using Remark \ref{fix0}.
Let $\OO$  a $\Wf$-orbit in $\tilde\X_*(\T)$ such that $\ell_\OO>0$.
The element  $\upiota_C(z_\OO)$ is central in $\Hh_\Z$. 
Recall that  a $\Z$-basis for $\Zz(\Hh_\Z)$ is given by the central Bernstein functions $z_\OO$ where $\OO$ ranges over the $\Wf$-orbits in $\tilde\X_*(\T)$.  
Lemma  \ref{decomZ} implies that $\upiota_C(z_\OO)$ decomposes as a sum
$$\upiota_C(z_\OO)= z_{\OO}+\sum_{\OO'}a_{\OO'} z_{\OO'}$$ where $\OO'$ ranges over a finite set of  $\Wf$-orbits in $\tilde\X_*(\T)$ such that  $\ell_{\OO'}<\ell_\OO$ and $a_{\OO'}\in \Z$. By induction and applying  the involution $\upiota_C$ we get
$$z_\OO= \upiota_C(z_{\OO})+\sum_{\OO'}a_{\OO'}z_{\OO'}$$ 
and  $2(\upiota(z_\OO)-z_{\OO}))=0$. Since $\Hh_\Z$ has no $\Z$-torsion, $\upiota(z_\OO)=z_{\OO}$. The second statement follows from \eqref{involution}.

\end{proof}

When $\Gp$ is semisimple, the projection  in $\Hh_k$ of the equality proved in  Proposition \ref{prop:fix} can be obtained independently  using  the duality for finite length $\Hh_{k}$-modules  defined in \cite{OS}:
\begin{prop} \label{theo:indep} Suppose that $\Gp$ is semisimple. 
   The element
$\sum_{\lambda'\in \OO} \Bb_{C}^\sigma(\lambda')\in \Hh_{k}$ is fixed by the involution $\upiota_C$  and therefore does not depend on the
sign $\sigma$.

\end{prop}

\begin{proof} 
Suppose that  $\Gp$ is semisimple.  Let $\OO$ be a $\Wf$-orbit in 
$\tilde\X_*(\T)$.  We want to prove, without using Proposition \ref{prop:fix}, that in $\Hh_k$ we have
$z_\OO=\upiota_C(z_\OO)$.\\
 Let $\zeta: \Zz(\Hh_k)\rightarrow k$ a character and $M=\Hh_k\otimes_{\Zz(\Hh_k)} \zeta$ the induced $\Hh_k$-module. It is  finite dimensional over $k$  and therefore, by \cite[Corollary 6.12]{OS} we have an isomorphism of right $\Hh_k$-modules
$${\rm Ext}^d_{\Hh_k}(M, \Hh_k)=\Hom_{k}(\upiota_C^*M, k)$$
where $d$ is the semisimple rank of $\Gp$ and $\upiota_C^*M$ denotes the left $\Hh_k$-module $M$ 
with action twisted by the 
involution $\upiota_C$  defined by \eqref{therightinvolution}.
The category of left $\Hh_k$-modules is naturally a $\Zz(\Hh_k)$-linear category and therefore, for  $X$ and $Y$ two given  left $\Hh_k$-modules, ${\rm Ext}^d_{\Hh_k}(X,Y)$ inherits  a structure of 
central $\Zz(\Hh_k)$-bimodule.
Hence, the right $\Hh_k$-module ${\rm Ext}^d_{\Hh_k}(M, \Hh_k)$ has  a central character equal to $\zeta$. On the other hand,
$\Hom_{k}(\upiota_C^*M, k)$ has $\zeta\circ \upiota_C$ as a central character. 
Therefore, 
$\zeta(z_{ \OO})=\zeta\circ \upiota_C(z_{ \OO})$. 
By Corollary \ref{coro:char}, we have  the required equality $z_{ \OO}=\upiota_C(z_{ \OO})$.
\end{proof}

\subsection{Independence  lemma} The following lemma will be proved in
\ref{proofbyinduction}.
\begin{lemma} For  $\OO$  a $\Wf$-orbit in 
$\tilde\X_*(\T)$,
the element \begin{equation*}
\sum_{\lambda\in \OO} \Bb_{F}^\sigma(\lambda)\end{equation*}  in $\Hh_\Z$ does not depend on the choice of the standard facet $F$ and of the sign $\sigma$.
\label{theoA}
\end{lemma}

\begin{coro} The center of $\Hh_\Z$ is contained in the intersection of all the commutative rings $\Aa_F^\sigma$ for $F$ a standard facet and $\sigma$ a sign.
\end{coro}
%\noindent In the case of $\G={\rm GL}_3(\Corps)$, there are $4$ standard facets:  the vertex $x_0$, the chamber $C$ and two edges $F_1$ and $F_2$. 
%\begin{figure}[h!]
%\begin{center}
%\includegraphics[width=7cm]{ellipse.pdf}
 %\end{center}
%\end{figure}
%[height=12cm]

\subsection{Inducing the generalized integral Bernstein functions\label{proof:theo2}}
We study the behavior of the integral Bernstein maps upon  parabolic induction and subsequently prove Lemma \ref{theoA}.

\subsubsection{\label{induc}} Let $F$ be a standard facet, $\Pi_F$ the associated set of simple roots and $\P_F$
the  corresponding standard parabolic subgroup with Levi decomposition $\P_F=\M_F{\rm N}_F$.  The root datum attached to the choice of the split torus $\T$ in $\M_F$ is  $(\Phi_F, {\rm X}^*(\Tp), \check\Phi_F, {\rm X}_*(\Tp))$ (notations in \ref{cones}).
 The extended Weyl group of $\M_F$ is
$\W_F=(N_\Gp(\T)\cap \M_F)/\Tp^0$. It is isomorphic to the semidirect product  $\Wf_F\ltimes \X_*(\T)$ where $\Wf_F$ is the finite Weyl group  $(N_\Gp(\T)\cap \M_F)/\Tp$ (also defined in \ref{cones}). We denote by $\ell_F$ its length function and by $\underset {F}\leq $ the Bruhat order on $\W_F$.

Set $\tilde \W_F=(N_\Gp(\T)\cap \M_F)/\Tp^1$. It is a subgroup of $\tilde \W$. 
The double cosets of $\M_F$ modulo its pro-$p$ Iwahori subgroup $\I\cap \M_F$   are indexed by the elements in  $\tilde \W_F$.
For $w\in \W_F$, we denote by $\t_w^F$ the characteristic function of the  double coset containing  the lift $\hat w$ for $w$ (which lies in $ N_\G(\T)\cap \M_F$). 
The set of all 
$(\tau^F_w)_{w\in \W_F}$ is a basis for the pro-$p$ Iwahori-Hecke ring 
$\Hh_\Z(\M_F)$ of  $\Z$-valued functions with compact support in $(\I\cap \M_F)\backslash \M_F/(\I\cap \M_F)$.  
The ring $\Hh_\Z(\M_F)$  does not inject in $\Hh_\Z$ in general.

\medskip

An element in $w\in\W_F$  is called $F$-positive if 
$w^{-1}(\Phi^+-\Phi_F^+)\subset \Phi_{aff}^+.$ For example for $\lambda\in\X_*(\T)$, the element $e^\lambda$ is  $F$-positive
if  and only if $\lp \lambda, \alpha \rp\geq 0$  for all $\alpha\in \Phi^+-\Phi_F^+.$ In this case, we will say that the coweight $\lambda$ itself is $F$-positive.  If furthermore   $\lp \lambda, \alpha \rp>0$  for  $\alpha\in \Phi^+-\Phi_F^+$ and 
$\lp \lambda, \alpha \rp= 0$  for  $\alpha\in \Phi_F^+$, then it is called strongly $F$-positive. The $F$-positive coweights 
are the $\Wf_F$-conjugates of the dominant coweights. 
The $C$-positive (resp.  strongly $C$-positive) coweights  are the dominant (resp. strongly dominant) coweights.
An element in $\W_F$ is $F$-positive if and only if it belongs to $ e^\lambda \Wf_F$ for some $F$-positive coweight $\lambda\in \X_*(\T)$.  If $\mu$ and $\nu \in \X_*(\T)$ are $F$-positive  coweights such that $\mu-\nu$  is also $F$-positive, then we have the equality (see \cite[1.2]{invsatake} for example)
\begin{equation}\label{lengthF}
\ell(e^{\mu-\nu})+\ell(e^{\nu})-\ell(e^{\mu})=  \ell_F(e^{\mu-\nu})+\ell_F(e^{\nu})-\ell_F(e^{\mu})
\end{equation}
An element in $\tilde\W_F$ will be called $F$-positive if its projection in $\W_F$ is $F$-positive.

\medskip

 The subspace of   $ \Hh_{\Z}(\M_F)$ generated over $\Z$ by all $\t_w^F$ for $F$-positive $w\in \tilde\W_F$ 
 is denoted by   $ \Hh_{\Z}(\M_F)^+$. It is in fact a ring and there is an injection of rings 
$$\begin{array}{cccc}j_F^+:& {\Hh}_{\Z}(\M_F)^+&\longrightarrow& {\Hh}_{\Z}\cr &\t^F_w&\longmapsto &\t_w\cr\end{array}$$ which extends to an injection of $\Z[q^{\pm 1/2}]$-algebras
$$j_F:{\Hh}_{\Z}(\M_F)\otimes _\Z \Z[q^{\pm 1/2}]\rightarrow  {\Hh}_{\Z}\otimes_\Z \Z[q^{\pm 1/2}].$$
This is a classical result for  complex Hecke algebras (\cite[(6.12)]{BK}). The argument is valid over $\Z[q^{\pm 1/2}]$.

\begin{rema} An element  $w\in\tilde\W_F$  is called  $F$-negative  (resp. strongly $F$-negative) if 
 $w^{-1}$ is $F$-positive (resp. strongly $F$-positive) and as before, $ \Hh_{\Z}(\M_F)$ contains as a subring  the space   $ \Hh_{\Z}(\M_F)^-$ generated over $\Z$ by all $\t_w^F$ for $F$-negative $w\in \tilde\W_F$. There is an injection of rings
 $j_F^-: {\Hh}_{\Z}(\M_F)^-\longrightarrow {\Hh}_{\Z}, \t_w^F\longmapsto \t_w.$
\label{rema:nega}
\end{rema}
\begin{fact} Let $v\in \W_F$ such that $v\underset{F}\leq e^\lambda$ for $\lambda\in \X_*(\T)$ a $F$-positive coweight. Then $v$ is $F$-positive. 
\label{fact1}
\end{fact}
 
\begin{proof} Suppose first that $\lambda$ is dominant. Then the claim  is \cite[Lemma 2.9.ii]{invsatake}.
In general,  $\lambda$  is a $\Wf_F$-conjugate of a dominant coweight $\lambda_0$:  there is $u\in \Wf_F$ such that
$e^\lambda= ue^{\lambda_0} u^{-1}$. We argue by induction on $\ell_F(u)$. Let $s$ be a simple reflection in $\Wf_F$ such that $\ell_F(su)= \ell_F(u)-1$.   By the properties of Bruhat order (see \cite[Lemma 4.3]{Haines} for example), one of  $v$, $vs$, $sv$, $svs$ is  $\underset{F}\leq se^\lambda s$ and by induction  this element is $F$-positive, which implies that $v$ is  $F$-positive.

\end{proof}
\subsubsection{\label{basestep}}

 Let $F'\subseteq\overline C$ be another facet containing $x_0$ in its closure such that
$F\subseteq \overline F'$. It implies that  $\Phi_{F'}\subseteq \Phi_F$ and $\Phi^+_{F'}\subseteq\Phi^+_F$. 
Let  ${}_F\thg_{F'}^+$ be   the  map constructed
as in \ref{modified} with respect to the root data attached to $\M_F$:  $${{}_F}\thg_{F'}^{+}:\Z[q^{\pm 1/2}][\tilde\X_*(\T)]\longrightarrow   \Hh_\Z(\M_F)\otimes_\Z \Z[q^{\pm 1/2}].$$ The corresponding $\Z$-linear integral map is denoted by ${}_F\Bb_{F'}^+: \Z[\tilde\X_*(\T)]\longrightarrow   \Hh_\Z(\M_F)$ and defined by 
${}_F\Bb_{F'}^+(\lambda)= q^{\ell_F(e^\lambda)/2}\:{{}_F}\thg_{F'}^{+}(\lambda)$ for all $\lambda\in \tilde \X_*(\T).$ It satisfies $\bf(\lambda)=\t_{e^\lambda}^F$ if  $\lp\lambda, \alpha\rp\geq 0$ for all 
$\alpha\in (\Phi_F^+-\Phi_{F'}^+)\cup \Phi_{F'}^-$.

\begin{rema} If $F=x_0$ then ${}_{x_0}\Bb^+_{F'}=\Bb_{F'}^+$.

\end{rema}

 \begin{lemma}
 
\label{lemma:jFanti}

Let $\lambda\in \tilde\X_*(\T)$ be an $F$-positive coweight. Then
$\bf(\lambda)$ lies in
$\tilde{\rm H}_{\Z}(\M_F)^+$
 and 
\begin{equation}\label{induction}
j_F^+(\bf(\lambda))= \Bb_{F'}^+(\lambda).
\end{equation}
\end{lemma}

\begin{proof}
Decompose $\lambda= \mu-\nu$ with $\mu,\nu\in \Cute^+(F')$. 
 Then in $\Hh_\Z(\M_F)\otimes_\Z \Z[q^{\pm 1/2}]$ we have $\bf(\lambda)= q^{(\ell_F(e^\lambda)+\ell_F(e^\nu)-\ell_F(e^\mu))/2}\t^F_{e^{\mu}} (\t^F_{e^{\nu}})^{-1}$.  By  Lemma \ref{lemma:fundam} applied to 
  the pro-$p$ Iwahori-Hecke algebra $\Hh_\Z(\M_F)$, it
  decomposes in $\Hh_\Z(\M_F)$  into a linear combination of $\t_{\tilde w}^F$ for $\tilde w\in \tilde\W_F$ where the projection  $w$ of $\tilde w$ in $\W_F$ satisfies $w\underset{F}\leq e^\lambda$.
  Fact  \ref{fact1} ensures that those $w$ (and $\tilde w$) are $F$-positive.
 Now, $j_F$ respects the product  and $$j_F^+(\bf(\lambda))= j_F(\bf(\lambda))= q^{(\ell_F(e^\lambda)+\ell_F(e^\nu)-\ell_F(e^\mu))/2}\t_{e^{\mu}} (\t_{e^{\nu}})^{-1}$$ because $\mu$ and $\nu$ are in particular $F$-positive.  Apply  \eqref{lengthF} to conclude.

\end{proof}
\subsubsection{\label{proofbyinduction}} 
We prove Lemma \ref{theoA}. Let $\OO$ be a $\Wf$-orbit in $\tilde\X_*(\T)$.
Since  $\Bb_{x_0}^+=\Bb_C^-$ and using \eqref{involution},
it is enough to prove 
\begin{equation}\label{toprove}\sum_{\lambda\in \OO} \Bb_{F}^+(\lambda)=\sum_{\lambda\in \OO} \Bb_{C}^+(\lambda)\end{equation}
for any standard facet $F$.
If $F= x_0$ then the result is given by Proposition \ref{prop:fix}. 
Let $F$ be a standard facet such that
$F\neq x_0$. \\\\ %Then the associated Levi subgroup
%is a proper Levi subgroup of $\Gp$ and its semisimple rank is strictly smaller than the semisimple rank of $\Gp$.\\\\
1/  Let  $ \mu\in \tilde\X_*(\T)$ be  a $F$-positive coweight with $\mathfrak W_F$-orbit   $\OO_F$. We have the following identity
 $$\sum_{\mu'\in \OO_F}\Bb_F^+(\mu')=
 \sum_{\mu'\in \OO_F}j_F^+
({}_{F}\Bb_F^+(\mu'))=
 \sum_{\mu'\in \OO_F}j_F^+({}_{F}\Bb_C^+(\mu'))= \sum_{\mu'\in \OO_F}\Bb_C^+(\mu')$$ where the first and third equalities come from \eqref{induction} and the second one from 
 Proposition \ref{prop:fix} applied to $\M_F$.\\\\
 2/   
 Choose $\nu$ a strongly  $F$-positive  coweight such that $\lambda+\nu$ is $F$-positive for all $\lambda\in \OO$.
 Decompose  the $\Wf$-orbit $\OO$ into the disjoint union of $\mathfrak W_F$-orbits   $\OO_F^{\,i}$ for $i\in\{1, .., r\}$.
  Since $\nu$ lies in both $\tilde\X_*^+(\T)$ and $\Cute^+(F)$,  we have $\Bb_F^+(-\nu)=\Bb_C^+(-\nu)=\upiota_C(\tau_{e^{-\nu}})$. \\\\
Let $i\in\{1, ..., r\}$
and  $\lambda\in \OO_F^{\,i}$. We have in $\Hh_\Z\otimes_\Z\Z[q^{\pm 1/2}]$
$$\Bb_F^+(\lambda)=q^{\frac{\ell(e^{\lambda})-\ell(e^{\lambda+ \nu})-\ell(e^{\nu})}{2}} \Bb_F^+(\lambda+\nu) \Bb_F^+(-\nu).$$\\
Note  that $\ell(e^{\lambda})-\ell(e^{\lambda+ \nu})-\ell(e^{\nu})$ does not depend on $\lambda\in \OO_F^{\,i}$: since $\lp\nu ,\alpha\rp=0$ for all $\alpha\in \Phi^+_F$,  this quantity is equal to
$\sum_{\alpha\in \Phi^+- \Phi_F^+}\vert \lp  \lambda, \alpha\rp\vert- \vert \lp  \lambda+ \nu, \alpha\rp\vert-
\vert \lp \nu, \alpha\rp\vert$ which does not depend on the choice of $\lambda\in \OO_F^{\,i}$ because $\Phi^+- \Phi_F^+$ is invariant under the action of $\mathfrak W_F$.
Therefore, if we pick a representative $\lambda_i\in  \OO_F^{\,i}$, we have
\begin{align*} \sum_{\lambda\in  \OO_F^{\, i}}\Bb_F^+(\lambda)&=q^{\frac{\ell(e^{\lambda_i})-\ell(e^{\lambda_i+ \nu})-\ell(e^{\nu})}{2}} \sum_{\lambda\in  \OO_F^{\, i}} \Bb_F^+(\lambda+\nu) \Bb_C^+(-\nu).\cr&=q^{\frac{\ell(e^{\lambda_i})-\ell(e^{\lambda_i+ \nu})-\ell(e^{\nu})}{2}} \sum_{\lambda\in  \OO_F^{\, i}} \Bb_C^+(\lambda+\nu) \Bb_C^+(-\nu)\:\:\qquad\textrm{by 1/ applied to the $\Wf_F$-orbit of $\lambda+\nu$}\cr&= \sum_{\lambda\in  \OO_F^{\, i}}\Bb_C^+(\lambda)\end{align*} which proves that 
$ \sum_{\lambda\in  \OO}\Bb_F^+(\lambda)= \sum_{\lambda\in  \OO}\Bb_C^+(\lambda)$.\\\\

 \section{Compatibility between Satake and Bernstein-type isomorphisms in characteristic $p$.\label{sec:compa}} 
\medskip

In this section all the algebras have coefficients in $k$.

\medskip

Let $(\rho, \V)$ be a weight and $v$ a chosen nonzero $\I$-fixed vector. Let $\chi: \H_k\rightarrow k$ be the associated character and $F_\chi$ the corresponding standard facet (Remark \ref{thefacet}). 
We consider the compact induction $\ind_\K^\Gp\rho$ and its $k$-algebra of $\Gp$-endomorphisms
$\cal H (\Gp, \rho)$. The $\I$-invariant subspace $(\ind_\K^\Gp\rho)^\I$  is naturally a right $\Hh_k$-module.
Let $\1_{\K, v}\in \ind_\K^\Gp \rho$ be the ($\I$-invariant) function with support $\K$ and value $v$ at $1$.
%By \cite[Lemma 3.6]{invsatake},  the map
%\begin{equation}\label{isofond0}\begin{array}{ccc}\chi\otimes _{\H_k} \Hh_k&\cong& (\ind_{\K }^\Gp\rho )^{\I}\cr 1\otimes 1&\mapsto& \1_{\K , v}\cr   \end{array}\end{equation} induces an
%$\Hh_k$-equivariant isomorphism. 
The map
\begin{equation}\label{satakecent}\begin{array}{ccc}
\Zz(\Hh_k)&\longrightarrow& \Hom_ {\Hh_k }((\ind_{\K }^\Gp\rho)^\I, (\ind_{\K }^\Gp\rho)^\I)\cr
z&\longmapsto  &[f \mapsto f z] \cr
\end{array}\end{equation} defines a morphism of $k$-algebras. 
On the other hand, by \cite[Corollary 3.14]{invsatake}, 
passing to $\I$-invariants yields 
an isomorphism  of $k$-algebras
\begin{equation}\label{bcn}\cal H (\Gp, \rho )= \Hom_ {\G }(\ind_{\K }^\Gp\rho, \ind_{\K }^\Gp\rho)\overset{\sim}\longrightarrow 
\Hom_ {\Hh_k }((\ind_{\K }^\Gp\rho)^\I, (\ind_{\K }^\Gp\rho)^\I).\: \end{equation}
Composing  \eqref{satakecent} with the inverse of \eqref{bcn} therefore gives a morphism of $k$-algebras  $\Zz(\Hh_k)\rightarrow \cal H (\Gp, \rho ) $ and we consider its restriction to $\Zz^\c(\Hh_k)$:

\begin{equation}\label{satakecenter}\begin{array}{ccl}
\Zz^\c(\Hh_k)&\longrightarrow& \cal H(\Gp, \rho)\cr
z&\longmapsto  &[\1_{\K, v} \mapsto \1_{\K, v} z] .\cr
\end{array}
\end{equation}
For $\lambda\in \X_*^+(\T)$, we denote by  
${\EuScript T}'_\lambda\in  \cal H(\Gp, \rho)$   the image by    \eqref{satakecenter}
of the central Bernstein function $z_{\lambda}$ defined by \eqref{zl}.

\bigskip

On the other hand, recall that we have the   isomorphism  of $k$-algebras (\cite[Theorem 4.11]{invsatake})
\begin{equation}\label{monsatake1}\begin{array}{cccl}
{\EuScript T}:&k[\X^+_*(\T)]&\overset{\simeq}\longrightarrow& \cal H(\Gp, \rho)\cr
%\lambda &\longmapsto  &{\EuScript T}_\lambda\cr
\end{array}
\end{equation}
where ${\EuScript T}_\lambda$ for $\lambda\in \X_*^+(\T)$ is defined by \begin{equation} \label{TB}{\EuScript T}_\lambda: \1_{\K, v}\mapsto \1_{\K, v}\Bb_{F_\chi}^+(\lambda).\end{equation}

\begin{prop}\label{prop:equal} We have ${\EuScript T}'_\lambda={\EuScript T}_\lambda$ for all $\lambda\in \X_*^+(\T)$.

\end{prop}

\begin{proof} It is enough to check that these operators coincide on 
$\1_{\K, v} $.  If $\lambda$ has length zero, then $\Bb_{F_\chi}^+(\lambda)=z_\lambda=\tau_{e^\lambda}$ and the claim is true.
 Otherwise $\lambda$  has length $>0$ and recall that $\OO(\lambda)$ denotes the $\Wf$-orbit of $\lambda$.\\
a/ Let $\lambda'\in\OO(\lambda)$ and suppose that $\lambda'\neq \lambda$.  By  \eqref{fact:annu},  we have
$\Bb_{F_\chi}^+(\lambda') \Bb_{F_\chi}^+(\lambda)=\Bb_{F_\chi}^+(\lambda) \Bb_{F_\chi}^+(\lambda')=0$  in $\Hh_k$.
It implies that ${\EuScript T}_\lambda(\1_{\K, v}\Bb_{F_\chi}^+(\lambda') )=0$ and therefore that $\1_{\K, v}\Bb_{F_\chi}^+(\lambda')=0$ by \cite[Corollary 6.5]{Parabind} that claims that $\ind_{\K}^\Gp \rho$ is a 
torsion-free $\cal H(\Gp, \rho)$-module.

\medskip

\noindent b/ 
By Lemma \ref{theoA}, we have
\begin{align*}{\EuScript T}'_\lambda(\1_{\K, v})&=  \1_{\K, v} \Bb_{F_\chi}^+(\lambda)+\sum_{\lambda' \in\OO(\lambda), \lambda'\neq \lambda} \1_{\K, v}  \Bb_{F_\chi}^+(\lambda')\cr
&= {\EuScript T}_\lambda(\1_{\K, v})+ \sum_{\lambda' \in\OO(\lambda), \lambda'\neq \lambda} \1_{\K, v}\Bb_{F_\chi}^+(\lambda')\cr&={\EuScript T}_\lambda(\1_{\K, v}) \textrm{ by a/}.\end{align*}

\end{proof}

\begin{rema} \label{rema:key}
By \cite[Lemma 3.6]{invsatake},  the map
\begin{equation}\label{isofond0}\begin{array}{ccc}\chi\otimes _{\H_k} \Hh_k&\cong& (\ind_{\K }^\Gp\rho )^{\I}\cr 1\otimes 1&\mapsto& \1_{\K , v}\cr   \end{array}\end{equation} induces an
$\Hh_k$-equivariant isomorphism. 
 Proposition \ref{prop:equal} combined with 
\eqref{isofond0} proves that for $\lambda\in \X_*^+(\T)$, the right actions of $z_\lambda$  and  of $\Bb^+_{F_\chi}(\lambda)$  on $1\otimes 1\in  \chi\otimes _{\H_k} \Hh_k$  coincide. This remark will be important for the classification of the simple supersingular  $\Hh_k$-modules in  \ref{subsec:classi}.\end{rema}
Proposition \ref{prop:equal} implies:
\begin{theorem}\label{coro:diag} 
The diagram \begin{equation}\begin{CD}k[\X_*^+(\T)] @>{\eqref{Zmap}}>>  \Zz^\c(\Hh_k)\cr
@|   @VV{\eqref{satakecenter} }V \cr
k[\X_*^+(\T)] @>{\EuScript{T}}>>        \cal H(\Gp, \rho)\cr
\end{CD}\label{diag}\end{equation}
 is a commutative diagram of isomorphisms of $k$-algebras. 
%\begin{center} \hspace{1cm}  \xymatrix{ & \ar[ld]_*[@]{\sim}^{\eqref{Zmap}}k[\X_*^+(\T)]\ar[rd]_*[@]{\sim}^{\EuScript{T}}& \\ \Zz^\c(\Hh)  \ar[rr]^*[@]{\sim}_{\eqref{satakecenter} }& & \cal H(\Gp, \rho)\\ }\end{center}
\end{theorem}

Remark that we have not used the fact that {\eqref{Zmap}} is multiplicative. We proved this fact beforehand
in Proposition \ref{prop:Zmap} but it can also be seen as a consequence of the commutativity of the diagram.

\section{Supersingularity\label{sec:supersing}}

We turn to the study of the $\Hh_k$-modules with finite length.
We consider  right   modules unless otherwise specified. Recall that $k$ is algebraically closed with characteristic $p$.

\subsection{A basis for the pro-$p$ Iwahori-Hecke ring\label{intbasis}} We recall the $\Z$-basis for $\Hh_\Z$ defined in  \cite{Ann}. It is indexed by  $w\in\tilde \W$ and is denoted by $(E_w)_{w\in \tilde\W }$ in \cite{Ann}.
We  will call it $(\Bb_{x_0}^+(w))_{w\in \tilde\W }$ because it coincides on $\tilde\X_*(\T)$ with the definition introduced in \ref{modified} (see also Remark \ref{compa}).  Recall  that we have a decomposition of $\tilde \W$ as the semidirect product:
$$\tilde \W=\X_*(\T)\rtimes \tilde\Wf.$$  For $w_0\in \tilde \Wf$ set $\Bb_{x_0}^+(w_0)=\tau_{w_0}$ and for 
$w= e^\lambda w_0\in  \X_*(\T)\rtimes \tilde\Wf $,  define  in $\Hh_\Z\otimes_\Z \Z[q^{\pm 1/2}]$:
$$\Bb^+_{x_0}(w)=q^{(\ell(w)-\ell(w_0)-\ell(e^\lambda))/2}\Bb_{x_0}^+(\lambda) \Bb_{x_0}^+(w_0)= q^{(\ell(w)-\ell(w_0))/2}\Theta_{x_0}^+(\lambda)\tau_{w_0}.$$     
By \cite[Theorem 2 and Proposition 8]{Ann}, this element 
lies in $\Hh_\Z$
and  the set of all $(\Bb_{x_0}^+(w))_{w\in \tilde\W }$ is a $\Z$-basis for $\Hh_\Z$.

\begin{rema}\label{directsummand}
As a $\Z$-module,   $\Hh_\Z$ is the direct sum of $\Aa_{x_0}^+$ and of the $\Z$-module with basis
$(\Bb_{x_0}^+(e^\lambda w_0))$ where $\lambda$ ranges over $\X_*(\T)$ and
$w_0$ over the set of elements in  $\tilde \Wf$ the projection of which in $\Wf$ is nontrivial.
Applying \eqref{involution}, we obtain that  the $\Z$-module $\Aa_{C}^+$  is a direct summand of $\Hh_\Z$ as well.

\end{rema}

\begin{rema} Let $d\in \Dd$ and $\tilde d\in\tilde \W$ a lift for $d$.
Write  $\tilde d= e^{\lambda} w_0$ with $w_0\in \tilde \Wf$, $\lambda\in \X_*^+(\Tp)$ and $\ell(e^{\lambda})=\ell(d)+\ell(w_0)$ (Proposition \ref{prop:dist}).   
%(which implies $\tau_{e^{\lambda}}=\tau_{d^{-1}}\tau_{w_0}$).
Then in $\Hh_\Z\otimes_\Z\Z[q^{\pm 1/2}]$, we  have \begin{equation}\label{Bd}\Bb_{x_0}^+(\tilde d)=q^{(\ell(\tilde d)-\ell(w_0)+\ell(e^\lambda))/2}\tau_{e^{-\lambda}}^{-1}\tau_{w_0}= q^{\ell(\tilde d)}\tau_{\tilde d^{-1}}^{-1}=(-1)^{\ell(d)}\upiota(\tau_{\tilde d}).\end{equation}
\end{rema}

\subsection{\label{topo}Topology on the pro-$p$ Iwahori-Hecke algebra in characteristic $p$}
We consider the (finitely generated) ideal $\mathfrak I$ of $\Zz^\circ(\Hh_k)$ generated by all $z_\lambda$    for $\lambda\in \X_*^+(\T)$ such that $\ell(e^\lambda)>0$ and the associated   ring filtration of  $\Zz^\circ(\Hh_k)$.  A $\Zz^\circ(\Hh_k)$-module $\m$ can be endowed with the $\mathfrak I$-adic topology induced by the filtration
$$\m\supseteq \m  \mathfrak I\supseteq  \m \mathfrak I^2 \supseteq ...$$
An example of such  module is $\Hh_k$ itself.  
We define on $\Hh_k$ another  decreasing filtration $(F_n\Hh_k)_{n\in \N}$  by  $k$-vector spaces where 
\begin{equation}\label{fil}F_n\Hh_k:= \textrm{$k$-vector space generated by all $\Bb_{x_0}^+(w)$ for  $w\in \tilde\W$ such that $\ell(w)\geq n$.}\end{equation}

\begin{lemma} \label{lemma:stronger}The filtration \eqref{fil} is a filtration of $\Hh_k$ as a left $\Aa_{x_0}^+$-module. In particular, it is a filtration of
 $\Hh_k$ as a (left and right) $\Zz^\circ(\Hh_k)$-module.
It  is compatible with the $\mathfrak I$-filtration:
for all $n\in \N$, we have
$$(F_n \Hh_k) \, \mathfrak I= \mathfrak I\, (F_n \Hh_k) \:    \subseteq F_{n+1} \Hh_k.$$
\end{lemma}

\begin{proof} Let $\lambda\in \tilde \X_*(\T)$ and $w\in \tilde \W$.
From the definition of $ \Bb_{x_0}^+$, we see that 
$$\Bb_{x_0}^+(\lambda)\Bb_{x_0}^+(w)=q^{(\ell(e^{\lambda})+\ell(w)-\ell(e^\lambda w))/2} \Bb_{x_0}^+({e^\lambda}w)$$ and therefore, in $\Hh_k$ we have: $ \Bb_{x_0}^+(\lambda)\Bb_{x_0}^+(w)=0$ if $\ell(e^\lambda)+\ell(w)>\ell(e^\lambda w)$ and 
$\Bb_{x_0}^+(\lambda)\Bb_{x_0}^+(w) =\Bb_{x_0}^+({e^\lambda}w)$ if $\ell(w)+\ell(e^\lambda)=\ell(e^\lambda w)$.
 It proves the claims.

\end{proof}

\begin{prop} \label{prop:equiv}The $\mathfrak I$-adic topology on $\Hh_k$ is equivalent to the topology on $\Hh_k$ induced by the filtration $(F_n\Hh_k)_{n\in \N}$. In particular, it is independent of  the choice of the uniformizer $\varpi$.

\end{prop}

\begin{proof}  
%The previous lemma ensures  that the $\mathfrak I$-adic topology on $\Hh_k$ is stronger than the topology induced by $(F_n\Hh_k)_{n\in \N}$. 
We have to prove that given $m\in\N$,  $m\geq 1$, there is $n\in \N$ such that  $F_{n}\Hh_k\subseteq  \mathfrak I^m \Hh_k  $.

\begin{fact} For $\lambda\in \X_*(\T)$ such that $\ell(e^\lambda)>0$ and for  $m\geq 1$, we have $\Bb_{x_0}^+((m+1)\lambda)\in \mathfrak I^m \Hh_k$.\label{m+1}
\end{fact}

\begin{proof}[Proof of the fact] We check that for  $m\in \N$ we have $\Bb_{x_0}^+((m+1)\lambda)= z_{\lambda}^m \Bb_{x_0}^+(\lambda)$.
Notice that  $\Bb_{x_0}^+(2\lambda)= \Bb_{x_0}^+(\lambda)\Bb_{x_0}^+(\lambda)=z_{\lambda}\Bb_{x_0}^+(\lambda)$ by  \eqref{fact:annu} and Lemma \ref{theoA}. 
Now let $m\geq 2$. We have  $\Bb_{x_0}^+((m+1)\lambda)= \Bb_{x_0}^+(m\lambda)\Bb_{x_0}^+(\lambda)=z_{\lambda}^m \Bb_{x_0}^+(\lambda)$ by induction.

\end{proof}

\begin{fact}  \label{fact:born}Let $m\geq 1$.
There is $A_m\in \N$ such that   for any  $\lambda\in \X_*(\T)$, if
 $\ell(e^\lambda)> A_m$ then $\Bb_{x_0}^+(\lambda)\in  \mathfrak I^m\Hh_k $.

\end{fact}

\begin{proof}[Proof of the fact] Let $\{z_{\lambda_1},\dots,z_{ \lambda_r}\}$  be a system of generators of   $\mathfrak I$ where $\lambda_1, \dots, \lambda_r\in  \X_*^+(\T)$. Let  $A_m:=
m\sum_{i=1}^r \ell(e^{\lambda_i})$.
Let $\lambda\in \X_*(\T)$ such that $\ell(e^\lambda)>0$. It is $\Wf$-conjugate to an element $\lambda_0\in \X^+_*(\T)$ and one can write $\lambda= w_0.\lambda_0$ with $w_0\in \Wf$ and
$\lambda_0=\sum_{i=1}^r a_i \lambda_i$ with $a_i\in \N$ (not all  equal to zero). If  $\ell(e^\lambda)=\ell(e^{\lambda_0})>A_m$, then there is $i_0\in\{1, ..., r\}$  such that $a_{i_0}> m$ and
$\Bb_{x_0}^+(\lambda)=\prod_{i=1}^r \Bb_{x_0}^+(a_i(w_0.\lambda_i)) \in \Bb_{x_0}^+((m+1)(w_0.\lambda_{i_0}))\Hh_k\subseteq \mathfrak I^m\Hh_k $ by Fact \ref{m+1}.
\end{proof}

 We know turn to the proof of the proposition.  Let $m\geq 1$. To  any $w_0\in \Wf$  corresponds,  by \cite[(1.6.3)]{VigGene},  a finite set $\X(w_0)$ of elements in $\X_*(\T)$ such that
$$\textrm{ for all  $\lambda\in \X_*(\T)$  there is $\mu\in \X(w_0)$ such that $\ell(e^\lambda w_0 )= \ell(e^{ \lambda-\mu})+\ell ( e^{\mu} w_0)$}.$$
Let $\tilde w\in \tilde \W$ with image $w_0$  by  the projection $\tilde \W\rightarrow\Wf$. Its image $w$ by $\tilde \W\rightarrow \W$ has the form $w=  e^\lambda w_0\in  \X_*(\T) \rtimes \Wf$ and
 there is $\mu\in \X(w_0)$ such that $\ell(w)= \ell(e^{ \lambda-\mu}) +\ell (e^\mu w_0)
$. Choose lifts $\widetilde{e^\mu w_0 }$  and
$\widetilde{e^{ \lambda-\mu}}$ in $\tilde \W$ for  ${e^{\mu}w_0}$  and
${e^{ \lambda-\mu}}$. The product $\widetilde{e^{ \lambda-\mu}} \widetilde{e^{\mu}w_0}$
differs from $\tilde w$ by an element in $\T^0/\T^1$ (which has length zero). Therefore, 
$\Bb_{x_0}^+(\tilde w)\in\Bb_{x_0}^+(\lambda-\mu)  \Hh_k$ (see the proof of Lemma \ref{lemma:stronger} for example). 
If $\ell(\tilde w)> A_m(w_0):=A_m+{\rm max}\{ \ell (e^{\mu'}w_0), \,\mu'\in \X(w_0)\}$ then $\ell(e^{\lambda-\mu})>A_m$ and 
$\Bb_{x_0}^+(\tilde w)\in    \mathfrak I^m \Hh_k$ by Fact \ref{fact:born}.
%$\Bb_{x_0}^+(w_0 e^\lambda)= \Bb_{x_0}^+(w_0 e^{\mu})\Bb_{x_0}^+(e^{ \lambda-\mu})$}$$
We have proved that  $n>  {\rm max}\{A_m(w_0), w_0\in \Wf\}$ implies $F_{n}\Hh_k\subseteq   \mathfrak I^m \Hh_k$.

 \end{proof}

%Recall that $Z$ is the connected center of $\Gp$.  We fix a character $\omega: Z\rightarrow k$.
%There is an embedding $k[Z/Z\cap \T^1]\rightarrow  \Zz^\c(\Hh_k)$ given by\red{ $g\mapsto \tau_{g^{-1}}$ ????}

\subsection{The category of  finite length modules  over the pro-$p$ Iwahori-Hecke algebra in characteristic $p$}
We consider the abelian category $\Mod_{fg}(\Hh_k)$
%, resp. $\Mod_{fg}^\omega(\Hh_k)$ 
of all $\Hh_k$-modules with finite length.
%,  all $\Hh_k$-modules, 
%all $\Hh_k$-modules with finite length and such that  for any $g\in Z$ the central element $\tau_g$ acts by $\omega(g^{-1})$. 

\medskip

For a $\Hh_k$-module, having finite length is  equivalent to being finite dimensional as a $k$-vector space  (\cite[5.3]{Durham} or \cite[Lemma 6.9]{OS}).
Therefore,  any irreducible $\Hh_k$-module is finite dimensional  and has a central character, and any module in  $\Mod_{fg}(\Hh_k)$ decomposes uniquely into a direct sum of indecomposable modules.

\def\id{\mathfrak M}

\subsubsection{The category of finite dimensional $\Zz^\c(\Hh_k)$-modules.} \noindent 
Let $\Mod_{fd}(\Zz^\c(\Hh_k))$ denote the category of finite dimensional $\Zz^\c(\Hh_k)$-modules. For  $\id$  a maximal ideal of $\Zz^\c(\Hh_k)$, we consider 
 the full subcategory  $$ {\id}-\Mod_{fd}(\Zz^\c(\Hh_k))$$ 
of the modules $\m$ of $\id$-torsion, that is to say such that
 there is $e\in \N$ satisfying  $ \m\id ^e=0$.
The category $\Mod_{fd}(\Zz^\c(\Hh_k))$  decomposes into the direct sum of all $ {\id}-\Mod_{fd}(\Zz^\c(\Hh_k))$ where $ {\id}$-ranges over the maximal ideals of $\Zz^\c(\Hh_k)$.

\subsubsection{Blocks of $\Hh_k$-modules with finite length}
For  $\id$  a maximal ideal of $\Zz^\c(\Hh_k)$, we say that a  $\Hh_k$-module with finite length  is a $\id$-torsion  module if its
restriction  to a $\Zz^\c(\Hh_k)$-module
lies in the subcategory $ {\id}-\Mod_{fd}(\Zz^\c(\Hh_k))$. We denote by
\begin{equation}\label{bloc} {\id}-\Mod_{fg}(\Hh_k)\end{equation} the full subcategory 
of $\Mod_{fg}(\Hh_k)$ whose objects are the  $\id$-torsion  modules. 

\begin{lemma} Let $\mathfrak M$ and $\mathfrak N$ be two maximal ideals of $\Zz^\c(\Hh_k)$. If there is a nonzero $\mathfrak M$-torsion module $M$ and a nonzero  $\mathfrak N$-torsion module $N$
such that   ${\rm Ext}_{\Hh_k}^r(\m, \n)\neq 0$ for some $r\geq 0$, then  $\mathfrak M=\mathfrak N$.

\end{lemma}

\begin{proof}
For any $\Hh_k$-modules $X$ and $Y$, the natural morphisms of algebras $\Zz^\c(\Hh_k)\rightarrow {\rm End}_{\Hh_k}(X)$
and  $\Zz^\c(\Hh_k)\rightarrow {\rm End}_{\Hh_k}(Y)$  equip $\Hom_{\Hh_k}(X, Y)$ with a structure of central $\Zz^\c(\Hh_k)$-bimodule. The  space  ${\rm Ext}_{\Hh_k}^r(\m, \n)$ is therefore naturally a  central $\Zz^\c(\Hh_k)$-bimodule. It is an $\mathfrak M$-torsion module and a  $\mathfrak N$-torsion module:
it is zero unless $\mathfrak M=\mathfrak N$.

\end{proof}

 Since  $\Zz^\c(\Hh_k)$ is a central finitely generated  subalgebra of $\Hh_k$, 
an indecomposable $\Hh_k$-module  with finite length  is a $\id$-torsion module for some    maximal ideal $\id$  of $\Zz^\c(\Hh_k)$. 

%In this case, if $m$ denotes the length of $\m$, we have (by induction on $m$)
%$\id^m\m=\{0\}.$ Consider the filtration
%\begin{equation} 0= \id^m\m \subseteq \id^{m-1}\m
 %\subseteq ..... \subseteq \id\m\subseteq \m.
%\end{equation}

%\begin{prop} Si $\m$ est indécomposable....
%Ce sont des egalités strictes et les quotients sont irreductibles ????

%\end{prop}

\begin{comment}
\begin{proof}

i. Let $\m$ be a $\Hh_k$-module with length $m$ and suppose that it has type $\zeta$. Suppose that $m\geq 2$ and let $\n$ be an irreducible submodule. Then $\m/\n$ has type $\zeta$ and
by induction,  ${\id}^{m-1} \m\subseteq \n$ so 
${\id}^{m} \m=0$.
Now suppose that there is  there is $e\in \N$ such that $\id^e \m=\{0\}$.  We want to prove that $\m$ has type $\zeta$.
It is enough to prove the claim for $\m$  irreducible:
let  $z\in \Zz^\c(\Hh_k)$, then $(z-\zeta(z))^e$ acts by zero on $\m$ which implies that $z$ acts by  $\zeta(z)$.

ii.  If it has length 1 or 2, then okay.
Let $z\in \Zz^\c(\Hh_k)$. 
There is a polynomial $P\in k[X]$ with degree $\leq$ the length of $\m$ and such that $P(z)\m=0$. If it has $2$ distinct roots, then
$\m$ is decomposable. So ... conclude...
\end{proof}

\end{comment}
\begin{rema} A $\Hh_k$-module with finite length $M$ lies in the block corresponding to some maximal ideal $\mathfrak M$ if and only if all the characters of $\Zz^\c(\Hh_k)$ contained in $M$  have kernel $\mathfrak M$.

\end{rema}

\begin{rema}
The blocks \eqref{bloc} are not indecomposable. They can for example be further decomposed \emph{via} the idempotents introduced in \ref{idempo}.
\end{rema}

\subsubsection{\label{supersing-block}The supersingular block}

\begin{defi}  \label{defi:sup-id}
 We call a maximal ideal $\id$ of  $\Zz^\circ(\Hh_k)$ supersingular if it contains the ideal  $\mathfrak I$ defined in \ref{topo}.
A character of $\Zz^\c(\Hh_k)$ is called  supersingular if its kernel is a supersingular maximal ideal of 
$\Zz^\circ(\Hh_k)$. 

\end{defi}

Given a character $\omega$ of the connected center $Z$ of $\Gp$, there is a unique  supersingular character $\zeta_\omega$ of  $\Zz^\c(\Hh_k)$ satisfying $\zeta_\omega(z_\lambda)=\omega(\lambda(\varpi))$ for any $\lambda\in \X_*^+(\T)$ with length zero.
A character  of the center of $\Hh_k$ is called  \enquote{null} in \cite{Ann} if it takes value zero at all central elements \eqref{Z1} for all $\Wf$-orbits $ \OO$ in $\tilde \X_*(\T)$ containing a coweight with length $\neq 0$.  

\begin{lemma} \label{SSZ} A character $\Zz(\Hh_k)\rightarrow k$ is  \enquote{null}  if and only if  its restriction to $\Zz^\c(\Hh_k)$ is a supersingular character in the sense of Definition \ref{defi:sup-id}.

\end{lemma}

\begin{proof}  Consider a character $\zeta: \Zz(\Hh_k)\rightarrow k$ whose restriction to $\Zz^\circ(\Hh_k)$ is supersingular. We want to prove that $\zeta$ is \enquote{null}.
The $\Hh_k$-module $\Hh_k\otimes _{\Zz(\Hh_k)} \zeta$ being finite dimensional, it contains a character $\hat{\zeta}$ for the commutative finitely generated $k$-algebra $(\Aa_{x_0}^+)_ k$ and the restriction of $\hat{\zeta}$ to 
$\Zz(\Hh_k)$  coincides with $\zeta$. 
%We look at the restriction of  $\hat{\zeta}$  to  $(\Aa_{x_0}^+)_ k$.
 Let $\lambda\in \X_*^+(\T)$ with  $\ell(e^{\lambda})\neq 0$;
by \eqref{fact:annu},  there is at most one $\Wf$-conjugate $\lambda'$ of $\lambda$ such that $\hat{\zeta}(\Bb_{x_0}^+(\lambda'))\neq 0$
and if there exists such a $\lambda'$, then $\hat{\zeta}(z_\lambda)=\zeta(z_\lambda)\neq 0$, which is a contradiction: we have proved that $\hat{\zeta}(\Bb_{x_0}^+(\lambda'))=0$ for all 
$\lambda'\in \X_*(\T)$ with  $\ell(e^{\lambda'})\neq 0$ which  implies that it is also the case for  $\lambda'\in \tilde\X_*(\T)$ with  $\ell(e^{\lambda'})\neq 0$. Therefore, $\zeta$ is \enquote{null}.

\end{proof}
A finite dimensional $\Hh_k$-module $\m$ with  central character is  called supersingular in \cite{Ann} if this central character is\enquote{null}. 
We extend this definition.

\begin{propdefi} A  finite length $\Hh_k$-module  is in the supersingular block and is called \emph{supersingular} if and only if equipped with the discrete topology, it is a continuous module for the $\mathfrak I$-adic topology on $\Hh_k$ or equivalently, for the topology induced by the filtration \eqref{fil}.\\
%The notion of supersingularity is independent of the choice of the uniformizer $\varpi$. 
\label{propdefi}
\end{propdefi}

\begin{proof}
An indecomposable  $\Hh_k$-module $M$ with finite length is in the supersingular block if and only if there is $m\geq 1$ such that
$ M\mathfrak I^m =\{0\}$. Then use Proposition \ref{prop:equiv}.
\end{proof}

%Note that a finite length module is supersingular if and only if its irreducible constituents are all supersingular.  

%\begin{rema}\label{rema:adjoint-ind} If $\mathbf G$ is of adjoint type or $\mathbf G= {\rm GL}_n$, then by   
%Corollary \ref{coro:adjoint} and the proof of Proposition \ref{prop:conjugacy}, the ideal $\mathfrak I$ and therefore the notion of supersingularity for the  finite length $\Hh_k$-modules   is independent of all the choices.

%\end{rema}

\medskip
\subsection{Classification of the simple supersingular  modules over the pro-$p$ Iwahori-Hecke algebra in characteristic $p$\label{subsec:classi}}

We establish this classification in the case where the root system of $\Gp$ is irreducible which we will suppose in  \ref{par:classif}. Until then the results are valid without further assumption on the root system.

\subsubsection{} Denote by $\Hh_k^{aff}$  the natural image in $\Hh_k$ of the affine Hecke subring $\Hh_\Z^{aff}$ of $\Hh_\Z$ defined in \ref{defi:rings}. We generalize \cite[Theorem 7.3]{Oparab}:

\begin{prop}  \label{charaff}A finite length $\Hh_k$-module   in the supersingular block contains a character for the  affine Hecke subalgebra $\Hh_k^{aff}$.

\end{prop}

\begin{proof} Let $M$ be a $\Hh_k$-module with finite length in the supersingular block. By  Proposition \ref{propdefi}, there is $n\in\N$ such that, for any $w\in \tilde \W$, if $\ell(w)> n$ then $M \Bb_{x_0}^+(w) =0$. Let $x\in M$ supporting a character  for $\H_k$ (see \ref{weights}) and
let $d\in \Dd$ with maximal length such that $x \Bb_{x_0}^+(\tilde d)\neq 0$ where $\tilde d\in \tilde \W$ denotes a lift for $d$ (the property  $x\Bb_{x_0}^+(\tilde d)\neq 0$ does not depend on the choice of the lift $\tilde d$). As in the proof of \cite[Theorem 7.3]{Oparab}, we prove that $x':=x\Bb_{x_0}^+(\tilde d)$ supports a character for $\Hh_k^{aff}$ which is  the $k$-algebra generated by all $\tau_t$ and all $\tau_{\tilde s}$ for $t\in \T^0/\T^1$ and $s\in S_{aff}$ with chosen lift $\tilde s\in \tilde \W$ (see paragraph \ref{defi:rings}). 
From the relations \eqref{braid} we get that $x'\tau_t=
x \tau_{dtd^{-1}} \Bb_{x_0}^+(\tilde d)$ is proportional to $x'$.
Now let $s\in S_{aff}$. If $\ell(ds)=\ell(d)-1$, then $ds\in \Dd$ after Proposition \ref{prop:dist} and, by \eqref{Bd}, the element  $x'$ is equal to $x\upiota(\tau_{\tilde d \tilde s}) \upiota(\tau_{\tilde s})$ (up to an invertible element in $k$), so $x'\tau_{\tilde s} =0$ by Remark \ref{rema:prodnul}.
If $\ell(ds)= \ell(d)+1$ 
and $ds\in \Dd$, then $x\Bb^+_{x_0}(\tilde d\tilde s)$ is equal to zero on one side and,   by  \eqref{Bd}, to  $x'\upiota(\tau_{\tilde s})$ (up to an invertible element in $k$)
 on the other side. It proves   that $x'\tau_{\tilde s}$ is proportional to $x'$ by Remark \ref{rema:prodnul}.  If $\ell(ds)= \ell(d)+1$ 
and $ds\not\in \Dd$
then there is $s'\in S$ such that $ds=s'd$ by Proposition \ref{prop:dist},
 and
$x'\upiota(\tau_{\tilde s})$ is proportional to $x\upiota(\tau_{\tilde s'}) \Bb_{x_0}^+(\tilde d)$ and therefore  to $x'$ because $ \upiota(\tau_{\tilde s'})\in\H_k$. We conclude that $x'\tau_{\tilde s}$ is proportional to $x'$  by Remark \ref{rema:prodnul}.

\end{proof}
 
 \subsubsection{Characters of $\Hh_k^{aff}$} 
 We call character of $\Hh_k^{aff}$ a morphism of $k$-algebras $\Hh_k^{aff}\rightarrow k$.
A character $\Xx$  of $\Hh_k^{aff}$ is completely determined by:
 \begin{itemize}
 \item[-] the unique $\xi\in  \hat{\overline {\mathbf T}}(\fq)$  such that $\Xx(\epsilon_\xi)=1$ (see notation in \ref{idempo}). This $\xi$ is 
 defined by $\xi(t)=\Xx(\tau_t)$ where $t\in\T^0/\T^1\simeq {\overline {\mathbf T}}(\fq) $ and we call it the restriction of $\Xx$ to $k[\T^0/\T^1]$.
 \item[-] the values $\Xx(\tau_{n_A})$ for all $A\in S_{aff}$, which,  by the quadratic relations \eqref{Q'}  satisfy:
$\Xx(\tau_{n_A})\in\{0, -1\}$ if $\xi$ is trivial on $\T_A$  and 
 $\Xx(\tau_{n_A})=0$ otherwise.
 \end{itemize}
 
Conversely,  one checks that any such  datum of $\xi\in \hat{\overline {\mathbf T}}(\fq)$ and values $\Xx(\tau_{n_A})$ for all $A\in S_{aff}$ satisfying the above  conditions defines a character $\Xx$ of $\Hh_{k}^{aff}$.

\begin{exa}\label{triv-sign}

  The  pro-$p$ Iwahori-Hecke ring  $\Hh_\Z$ is endowed with two natural morphisms of rings $\Hh_\Z\rightarrow \Z$ defined by
 $$
 \tau_{w}\mapsto q^{\ell(w)} \quad\text{and}\quad \tau_{w}\mapsto (-1)^{\ell(w)}.
$$ We denote by $\Xx_{triv}$ and $\Xx_{sign}$ the characters of $\Hh_k$ that they respectively induce, as well as their restrictions to characters of $\Hh^{aff}_k$. The former can be described by: 
$\xi=\1$ and $\Xx_{triv}(\tau_{n_A})=0$ for all $ A\in S_{aff}$; the latter by  $\xi=\1$ and $\Xx_{sign}(\tau_{n_A})=-1 $  for all $ A\in S_{aff}$.

 \end{exa}

Let $\Xx$ be a character of $\Hh_k^{aff}$ and $\xi$ the corresponding element in $\hat{\overline {\mathbf T}}(\fq)$. \begin{itemize}
 \item Let $\xi_0\in \hat{\overline {\mathbf T}}(\fq)$ and suppose that $\xi_0$ is trivial on $\T_\alpha$ for all $\alpha\in\Pi$. Then one can consider the twist $ (\xi_0)\Xx$ of $\Xx$ by
 $\xi_0$ in the obvious way. The restriction of $ (\xi_0)\Xx$ to $k[\T^0/\T^1]$ is the product $\xi_0 \xi$ and  $ (\xi_0)\Xx$ coincides with $\Xx$  on the elements of type $\tau_{n_A}$  for $ A\in S_{aff}$. 
 By \emph{twist of the character $\Xx$} we mean from now on a twist of $\Xx$ by an element in    $\hat{\overline {\mathbf T}}(\fq)$  that is  trivial on $\T_\alpha$ for all $\alpha\in\Pi$.
 \item  The involution $\upiota_C$ extends to an involution of the $k$-algebra $\Hh_k$. The composition $\Xx\circ \upiota_C$ is then also a character for $\Hh_k^{aff}$.   Note that $\Xx$ and $\Xx\circ \upiota_C$
 have the same restriction to $k[\T^0/\T^1]$ (Remark \ref{fix0}). Furthermore, if $\Xx(\tau_{n_A})=-1$ for some $A\in S_{aff}$, then 
$\Xx\circ \upiota_C(\tau_{n_A})=0$ (use Remark \ref{rema:prodnul}). For example, $\Xx_{triv}=\Xx_{sign}\circ \upiota_C$.

 \item There is an action of $\tilde \Omega$ by conjugacy on  $\tilde \W_{aff}$. Since the elements in $\tilde \Omega$ have length zero, this yields an action of $\tilde \Omega$ on  $\Hh_k^{aff}$ and its  characters. For $\omega\in   \tilde \Omega$, we denote by $\omega.\Xx$ the character $\Xx(\tau_{\omega^{-1}}\, .\,\tau_{\omega})$.

  \end{itemize}

\begin{lemma} A simple $\Hh_k$-module containing a twist of the character
 ${\mathcal X}_{triv}$ or of  the character ${\mathcal X}_{sign}$ 
of  $\Hh^{aff}_k$ is not supersingular. 
\label{notss}

\end{lemma}

\begin{proof} Let $\m$ a simple $\Hh_k$-module.  
Suppose that it contains a twist of the character
 ${\mathcal X}_{sign}$ supported by the nonzero vector $m\in \m$. In particular, $m$ supports the character of $\H_k$ parametrized  by  (a twist of) the trivial character of $ \hat{\overline {\mathbf T}}(\fq)$ and by the facet $C$ (see \ref{weights}). By Remark \ref{rema:key}, we have
 $$m\, z_{\lambda}= m \,\Bb_{C}^+(\lambda)$$ for all $\lambda\in \X_*^+(\T)$. 
There is $\omega\in \tilde \Omega$ and $w\in \tilde \W_{aff}$ such that  the
element $\lambda(\varpi^{-1}) \mod \T^1$ corresponds to
$w\omega\in \tilde\W$.  Since $\Bb_C^+(\lambda)= \tau_{\lambda(\varpi^{-1})}$, the element $m\Bb_C^+(\lambda)$    is equal to $(-1)^{\ell(w)} m\tau_{\omega}$ (up to multiplication by an element in $k^\times$) and we recall that $\tau_{\omega}$ is invertible in $\Hh_k$. We have proved that $m.z_{\lambda}\neq 0$ and $\m$ is not supersingular.

Now if $\m$ contains  a twist of the character
 ${\mathcal X}_{triv}$, then $\upiota_C^*\m$ contains a twist
 of the character
 ${\mathcal X}_{sign}$ and is not supersingular (notation in the proof of Proposition \ref{theo:indep}). By Proposition \ref{prop:fix},   it implies that $\m$ is not supersingular either.
\end{proof}

 \subsubsection{} 
 Consider  the image  of $\tilde \Omega$  in $\Hh_k$ via 
$\omega\mapsto \tau_{\omega}$. 
For  $\mathcal X$   a character of  $\Hh_k^{aff}$,
 denote by $\tilde\Omega_{\mathcal X}$ its  fixator  under the action  of $\tilde \Omega$. It obviously contains $\T^0/\T^1$ as a subgroup.
 We consider the set $\Pp$ of  pairs $(\Xx, \sigma)$ where $\Xx$ is a character of   $\Hh_k^{aff}$ and $(\sigma, \V_\sigma)$  an irreducible finite dimensional $k$-representation of $\tilde\Omega_\Xx$ (up to isomorphism) whose restriction to $\T^0/\T^1$ coincides with the inverse of the restriction of $\Xx$: for any $t\in \T^0/\T^1$ and $v\in \V_\sigma$, we have 
 $\sigma(t) v=\Xx(\tau_{t^{-1}})v$.

The set $\Pp$ is  naturally endowed with an action of $\tilde \Omega$: for $(\Xx, \sigma)\in \Pp$ and $\omega \in \tilde\Omega$, denote by $\omega.\sigma$ the representation of $\tilde\Omega_{\omega.\Xx}=\omega \tilde\Omega_{\Xx} \omega^{-1}$ naturally obtained by conjugating $\sigma$;  then $\omega.(\Xx, \sigma):=(\omega.\Xx, \omega.\sigma)\in \Pp$.

\medskip

 Let $(\Xx, \sigma)\in \Pp$. Consider  the subalgebra $\Hh_k(\mathcal X)$ of $\Hh_k$ generated by 
$k[\tilde\Omega_{\mathcal X}]$ and $\Hh_k^{aff}$. It is isomorphic to the twisted tensor product of algebras
$$ \Hh_k(\mathcal X)\simeq k[\tilde\Omega_{\mathcal X}]\otimes_{k[\T^0/\T^1]}\Hh_k^{aff}$$
where the product is given by $(\omega\otimes h)  (\omega'\otimes h')=\omega \omega'\otimes \tau_{\omega'}^{-1} h\tau_{\omega'} h'$. As a left $\Hh_k(\mathcal X)$-module, $\Hh_k$ is free with basis the set of all $\tau_\omega$ where $\omega$ ranges over a set of representatives of  the right cosets $\tilde\Omega_{\mathcal X}\backslash \tilde \Omega$.
The tensor product  $\sigma\otimes \Xx$ is  naturally a   right $\Hh_k(\mathcal X)$-module: the right action of $\omega\otimes h$ on $v\in \V_\sigma$ is given by $\Xx(h)\sigma(\omega^{-1})v$.
The right $\Hh_k(\mathcal X)$-module $\sigma\otimes \Xx$ is irreducible. 
As a $\Hh_k^{aff}$-module, it is isomorphic to a direct sum of copies of $\Xx$.

\begin{lemma} The isomorphism classes of the simple $\Hh_k$-modules containing a character for $\Hh_k^{aff}$ are represented by  the induced modules
$$\mm(\Xx, \sigma):= (\sigma\otimes \Xx)\otimes_{\Hh_k(\Xx)}\Hh_k$$ where
 $(\Xx, \sigma)$ ranges over the set of orbits in $\Pp$ under the action of $\tilde\Omega$.
 \label{classiSuper}
\end{lemma}

\begin{proof} 
First note that for any $\omega\in \tilde \Omega$,   the
$(\Hh_k^{aff}, \omega.\Xx)$-isotypic component of $\mm(\Xx, \sigma)$ is isomorphic to 
$\omega.\sigma\otimes\omega \Xx$ as a right ${\Hh_k(\omega.\Xx)}$-module.\\
1/ We check that a  $\Hh_k$-module of the form  $\mm(\Xx, \sigma)$ is irreducible. 
Restricted to $\Hh_{k}^{aff}$ it is  semisimple and isomorphic to a direct sum of $\Xx$ and of its conjugates. Therefore, a submodule $\mm$ of  $\mm(\Xx, \sigma)$ contains a  nonzero $(\Hh_k^{aff}, \omega.\Xx)$-isotypic vector for some $\omega\in \tilde\Omega$ and  after translating by $\tau_{\omega^{-1}}$, we see that $\mm$ contains a  nonzero $(\Hh_k^{aff}, \Xx)$-isotypic vector.
But the $(\Hh_k^{aff}, \Xx)$-isotypic component in $\mm(\Xx, \sigma)$ supports the irreducible representation $\sigma$ of  $k[\tilde\Omega_\Xx]$. Therefore $\mm=\mm(\Xx, \sigma)$. \\
2/  Let $\mm$ be a simple $\Hh_k$-module containing the character $\Xx$ of $\Hh_k^{aff}$. Its $(\Hh_k^{aff}, \Xx)$-isotypic component 
% is a module for $k[\tilde\Omega_\Xx]$ and 
contains  an irreducible (finite dimensional) representation  $\sigma$ of 
$k[\tilde\Omega_\Xx]$ which coincides with  the inverse of $\Xx$ on  $k[\T^0/\T^1]$. Therefore, and using 1/,  $\mm\simeq (\sigma\otimes \Xx) \otimes_{\Hh_k(\Xx)} \Hh_k $.\\
3/  Let $\omega\in \tilde\Omega$ and $(\Xx, \sigma)\in \Pp$.  The $(\Hh_k^{aff}, \Xx)$-isotypic component of 
$\mm(\omega.(\Xx, \sigma))$  contains the representation $\sigma$ of   $k[\tilde\Omega_\Xx]$. 
The simple $\Hh_k$-module $\mm(\omega.(\Xx, \sigma))$ is therefore isomorphic to 
$\mm(\Xx, \sigma)$ by 2/.\\
4/ Let $(\Xx, \sigma)$ and $(\Xx', \sigma')$  in $ \Pp$ and suppose that they induce isomorphic $\Hh_k$-modules. Looking at the restriction of the latter to $\Hh_{k}^{aff}$ we see that there is $\omega\in \tilde \Omega$ such that  $\Xx'= \omega.\Xx$.
%Now $\mm(\Xx', \sigma')=\mm(\omega.(\Xx, \omega^{-1}\sigma'))\simeq \mm(\Xx, \omega^{-1}\sigma')$ by 3/. 
Therefore, by 3/, 
$\mm(\Xx, \omega^{-1}\sigma')$ and $\mm(\Xx, \sigma)$ are isomorphic and looking at the restriction to the 
$(\Hh_k^{aff}, \Xx)$-isotypic component  shows that $\sigma'\simeq \omega.\sigma$. Therefore,  
$(\Xx', \sigma')$ and $(\Xx, \sigma)$ are conjugate.

\end{proof}

\subsubsection{Classification of the simple supersingular $\Hh_k$-modules when the root system of $\G$ is irreducible.\label{par:classif}} 
We generalize \cite[Theorem 5(1)]{Ann}-\cite[Theorem 7.3]{Oparab}.

\begin{theorem}  Suppose that the root system of $\G$ is irreducible. 
\label{theo:charaff}A simple $\Hh_k$-module is supersingular if and only if it contains a character  for $\Hh^{aff}_k$ that is different from a twist of $\Xx_{triv}$  or   $\Xx_{sign}$. 
\end{theorem}

\begin{rema} This proves in particular (if the root system of $\G$ is irreducible) that the notion of supersingularity  for Hecke modules does not depend on any of the choices made.

\end{rema}

\begin{proof}[Proof of Theorem \ref{theo:charaff}] 
We already proved in Proposition \ref{charaff} (without restriction on the root system of $\G$), that a simple supersingular module contains a character for $\Hh_k^{aff}$ and by Lemma  \ref{notss}, we know that this character is not a twist of $\Xx_{triv}$  or   $\Xx_{sign}$. 

Conversely, let $\mm$ be a simple $\Hh_k$-module containing  the character $\mathcal X$   for $\Hh_k^{aff}$ and suppose that $\Xx$ is not a twist of $\Xx_{triv}$ or  $\Xx_{sign}$.
We want to prove that $\mm$ is supersingular. Since,  by Proposition \ref{prop:fix}, it is equivalent to showing that $\upiota_C^*\mm$ is supersingular (notation in the proof of Proposition \ref{theo:indep}), we can  suppose (see the discussion before Lemma \ref{notss}) that $\mathcal X(\tau_{n_0})=0$  where $n_0$ was introduced in \ref{not:irreducibleroot}.

Let $m\in \mm$ a nonzero vector supporting $\mathcal X$.   Let $\chi$ be the restriction of $\mathcal X$  to $\H_k$ and $F_\chi$ the associated standard facet.   Suppose that $F_\chi= x_0$, then $\Pi_{\bar \chi}=\Pi_{\chi}=\Pi$  (notation in \ref{weights}) and $\mathcal X(\tau_{n_\alpha})=0$ for all $\alpha\in \Pi$. Since, by hypothesis, we also have $\mathcal X(\tau_{n_0})=0$,   the character $\mathcal X$ is  equal to $\mathcal X_{triv}$ up to twist. Therefore, $F_\chi\neq x_0$.
Let  $\lambda\in \X_*^+(\T)$  with $\ell(e^\lambda)>0$. By Remark \ref{rema:key}
 $$m.z_{\lambda}= m. \Bb_{F_\chi}^+(\lambda).$$
and since   $F_\chi\neq x_0$,  we have $m.z_{\lambda}=0$ by  Lemma \ref{calcul-n0}.
We have proved that $\Zz^\circ(\Hh_k)$ acts on $m$ and therefore on $\mm$ by a supersingular character.

\end{proof}

Let $\Pp^*$ denote the subsets of pairs $(\Xx, \sigma)$ in $\Pp$ such that $\Xx$ is different from a twist of $\Xx_{triv}$  or   $\Xx_{sign}$. 
It is stable under the action of $\tilde\Omega$. 
Lemma \ref{classiSuper} and Theorem \ref{theo:charaff}
 together give the following:

\begin{coro} Suppose that the root system of $\G$ is irreducible. The map $$(\Xx, \sigma)\mapsto \mm(\Xx, \sigma)$$ induces a bijection between
the $\tilde\Omega$-orbits of pairs $(\Xx, \sigma)\in \Pp^*$    and a
 system of representatives of the isomorphism classes of the  simple supersingular $\Hh_k$-modules.
\end{coro}

\subsection{Pro-$p$ Iwahori invariants of parabolic inductions and of special representations.} 

\subsubsection{} 
In this paragraph,   $\kb$ is an arbitrary field.
Let $F$ be a standard facet, $\Pi_F$ the associated set of simple roots and $\P_F$
the  group of $\Corps$-points of the corresponding standard parabolic subgroup with Levi decomposition $\P_F=\M_F{\rm N}_F$. We use the same notations as in \ref{induc}. The unipotent subgroup  ${\rm N}_F$ is generated by all the root subgroups $\Uu_\alpha$ for $\alpha\in \Phi^+-\Phi_F^+$.  Let ${\rm N}_F^-$ denote the opposite unipotent subgroup of $\G$.
The pro-$p$ Iwahori subgroup $\I$ has the following decomposition:
$$\I= \: \I_F^+\:\I_F^0\:\I_F^-\textrm{ where  $\I_F^+:=\I\cap {\rm N}_F$,  $\I_{F}^0:=\I\cap \M_F$, $\I_F^-:=\I\cap {\rm N}_F^-$}.$$
Recall that, by Remark \ref{rema:nega}, the subspace  $ \Hh_{\kb}(\M_F)^-$  of $ \Hh_{\kb}(\M_F)$ generated over $\kb$ by all $\t_w^F$ for all $F$-negative $w\in \tilde\W_F$ identifies with a sub-$\kb$-algebra of $\Hh_\kb$   \emph{via} the injection
$$\begin{array}{cccc}j_F^-:& {\Hh}_{\kb}(\M_F)^-&\longrightarrow& {\Hh}_{\kb}\cr &\t^F_w&\longmapsto &\t_w\cr\end{array}.$$
This endows   $\Hh_k$ with a structure of left module over $\Hh_{\kb}(\M_F)^-$.

\begin{prop} Let $(\upsigma, \V_\upsigma)$ be a smooth    $\kb$-representation of $\M_F$. Consider the parabolic induction $\Ind_{\P_F}^\Gp\upsigma$ and its $\I$-invariant subspace  $(\Ind_{\P_F}^\Gp\upsigma)^\I$.
There is a surjective morphism of right $\Hh_\kb$-modules
\begin{equation}\label{eq:surjective} \upsigma^{\I_F^0}\otimes _{ {\Hh}_\kb(\M_F)^-}\Hh_\kb \longrightarrow (\Ind_{\P_F}^\Gp\upsigma)^\I\end{equation}
sending $v\otimes 1$ to the unique $\I$-invariant function with support in $\P_F\I$ and value $v$ at $1_\Gp$.
\label{prop:invparab}
\end{prop}

\begin{rema} In the case of ${\mathbf G}= {\rm PGL}_n$ or ${\rm GL}_n$,
Proposition 5.2 in \cite{Oparab} implies that \eqref{eq:surjective} is an isomorphism.  This  result should be true for a general (split) ${\mathbf G}$, but we will only use the surjectivity here.
\end{rema}

The proposition follows from the discussion below. All the   lemmas  are proved in the next paragraph.
\medskip

\begin{lemma} Let $\Dd_F=\{d\in \Wf, \: d^{-1}\Phi_F^+\subseteq \Phi^+\}$. 
\begin{itemize} 
\item[i.] For $d\in \Dd_F$, we have $\P_F\I\hat{d}\,\I=\P_F \hat{d}\,\I$.
\item[ii.] The set of all  $\hat{d}\in \Gp$  for $d\in \Dd_F$ is a system of representatives of the double cosets 
$\P_F\backslash \Gp/\I$.

\item[iii.] For $d\in \Dd_F$, let $\I \hat{ d}\, \I=\coprod_y\I \hat{d} y$ be a decomposition into right cosets.
Then $$\P_F \hat{d}\,\I=\coprod_{y}\P_F \I \hat{d} y.$$

\item[iv.] Let $d\in \Dd_F$. By the projection $\P_F\twoheadrightarrow \M_F$, the image of  $\P_F\cap \hat {d} \,\I \hat {d} ^{-1}$ is $\I_F^0$.
\end{itemize}
\label{fact:coset} 
\end{lemma}

An element  $m\in \M_F$  contracts $\I_F^+$ and dilates $\I_F^-$ if it satisfies the conditions (see  \cite[(6.5)]{BK}):
\begin{equation} \label{cd}m \I_F^+ m^{-1}\subseteq \I_F^+, \qquad m^{-1} \I_F^- m\subseteq \I_F^-. \end{equation}

\begin{rema} 
This property of an element  $m\in \M_F$ only depends on
the double coset \\ $\I_F^0m  \I_F^0$.  Furthermore,
if $m\in \K\cap \M_F$ then $m \I_F^+ m^{-1}= \I_F^+$  and  $m^{-1}\I_F^- m= \I_F^-. $\label{rema:equa}
\end{rema}

\begin{lemma} 
Let $w\in\tilde\W_F$.
The element $\hat w$ satisfies \eqref{cd} if and only if $w$ is $F$-negative.

\label{fact:cd-nega}
\end{lemma}

 Let $(\upsigma, \V_\upsigma)$  as in the proposition. 
Let $v\in\V_\upsigma^{\I_F^0}$ and $d\in \Dd_F$.
By Lemma \ref{fact:coset} ii and iv,  the $\I$-invariant function  
$$f_{d,v}\in (\Ind_{\P_F}^\Gp\sigma)^\I$$ with support in $\P_F\hat  d\, \I$ and value $v$ at $\hat d$ is well-defined and the set of all 
$f_{d,v}$ form a basis of $(\Ind_{\P_F}^\Gp\sigma)^\I$, when $d$ ranges over $\Dd_F$ and $v$ over a basis of $\V_\upsigma^{\I_F^0}$.

\begin{lemma}\begin{itemize}

\item[i.]  Let     $w$  an $F$-negative element in $\tilde\W_F$.
Then $f_{1,v}\:.\:\tau_w= f_{1, \, v.\tau_w^F}.$

\item[ii.] We have
$f_{1, v}\: .\:\tau_{\hat d}= f_{d,v} .$

\end{itemize}
\label{lemma:compu-parab}
\end{lemma}

\subsubsection{Proof of the lemmas}
Recall that
given $\alpha\in\Phi$, the root subgroup $\Uu_\alpha$
is endowed with a filtration $\Uu_{(\alpha, k)}$ for $k\in \Z$ (see for example  \cite[I.1]{SS} or  \cite[4.2]{OS}) and that the product map 
\begin{equation}\label{rootsubgroupeq}
    \prod_{\alpha\in \Phi^-} \Uu_{(\alpha, 1)}\times  \Tp^1 \times  \prod_ {\alpha\in \Phi^+} \Uu_{(\alpha, 0)}\overset{\sim}\longrightarrow \I
\end{equation}induces a bijection, where the products on the left hand side are ordered in some arbitrary chosen way (\cite[Proposition I.2.2]{SS}).
The subgroup $\I_F^+$ (resp. $\I_F^-$) of $\I$ is generated by the image of
 $ \prod_{\alpha\in \Phi^+-\Phi_F^+} \Uu_{(\alpha, 0)}$ (resp. 
 $ \prod_{\alpha\in \Phi^--\Phi_F^-} \Uu_{(\alpha, 1)}$). The subgroup
 $\I_F^0$ of $\I$ is generated the image of
 $ \prod_{\alpha\in \Phi_F^-} \Uu_{(\alpha, 1)}\times\T^1\times \prod_{\alpha\in \Phi_F^+} \Uu_{(\alpha, 0)}$.

\begin{proof}[Proof of Lemma \ref{fact:coset}]  i. We have    $\P_F\I\hat{  d}\I=\P_F \I_F^-\hat{  d}\I$. But for $\alpha \in \Phi^+$, we have $\hat{  d}^{-1} \,\Uu_{(-\alpha, 1) }\,\hat{  d}=
\Uu_{(-d^{-1}\alpha, 1)}\subseteq \I$ so $\I_F^-\hat{  d}\subseteq \hat{  d}\I$ and $\P_F\I\hat{  d}\I=\P_F\hat{  d}\I$. Point ii   follows by Bruhat decomposition for $\K$ and Iwasawa decomposition for $\Gp$. For iii,  
we first recall that the image of $\P_F\cap \K$ by the reduction $red:\K\rightarrow  \overline{\mathbf G}_{x_0}(\mathbb F_q)$ modulo $\K_1$  is  a parabolic
subgroup $ \overline{\mathbf P}_F(\mathbb F_q)$ containing $ \overline{\mathbf B}(\mathbb F_q)$ (notations in \ref{nots}).
% and with unipotent radical denoted by $\overline{\mathbf N}_F(\mathbb F_q)$. 
Recall that the Weyl group of $\overline{\mathbf G}_{x_0}(\mathbb F_q)$ is $\Wf$:
for $w\in \Wf$ we will still denote by $w$ a chosen lift in $\overline{\mathbf G}_{x_0}(\mathbb F_q)$.
The set $\Dd_F$ is a system of representatives of  $\overline{\mathbf P}_F(\mathbb F_q)\backslash \overline{\mathbf G}_{x_0}(\mathbb F_q)/\overline{\mathbf N}(\mathbb F_q)$.
For $d\in \Dd_F$ we have, using \cite[2.5.12]{Carter},
$$ \overline{\mathbf P}_F(\mathbb F_q)\cap d  \overline{\mathbf N}(\mathbb F_q) d^{-1}\subset \overline{\mathbf N}(\mathbb F_q). $$
 We deduce that  the image of $\P_F\cap \I_F^-\hat d \I \hat d^{-1}$ by $red$ is contained in $\overline{\mathbf N}(\mathbb F_q)$
and therefore $\P_F\cap  \I_F^- \hat d \I \hat d^{-1}$ is contained in $ \I$.\\
Now let $d\in\Dd_F$ and $y\in \I$. By the previous observations,  $\hat{  d}\in \P_F \I \hat{  d} y = \P_F \I ^-_F\hat{  d} y$  implies 
$\hat{  d}\in \I \hat{  d} y $. It proves iii.
In passing we proved that $\P_F\cap \hat d \I  \hat d^{-1}$ is contained in $\P_F\cap\I= \I_F^0\I_F^+$. Since  $ \I_F^0$ is contained in 
$\P_F\cap \hat d \I \hat d^{-1}$ by definition of $\Dd_F$, it proves iv.
\end{proof}

\begin{proof}[Proof of Lemma \ref{fact:cd-nega}] By Remark \ref{rema:equa} it is enough to prove the result for $w=e^\lambda\in\X_*(\T)$. A lift for $e^\lambda$ is given by $\lambda(\varpi^{-1})$.
The element $\lambda(\varpi^{-1})$ satisfies \eqref{cd} if   \begin{equation}\textrm{for all $\alpha\in \Phi^+-\Phi^+_F$ we have }\lambda(\varpi^{-1})\,\Uu_{(\alpha, 0)}\lambda(\varpi)\subseteq \I_F^+\textrm{ and }
\lambda(\varpi)\Uu_{(-\alpha, 1)}\lambda(\varpi^{-1})\subseteq \I_F^-.\label{cd'}\end{equation}
By \cite[Remark 4.1(1)]{OS} (for example),
  $\lambda(\varpi^{-1})\Uu_{(\alpha, 0)}\lambda(\varpi)= \Uu_{(\alpha, -\lp\alpha,\lambda\rp)}$
 and $ 
\lambda(\varpi)\Uu_{(-\alpha, 1)}\lambda(\varpi^{-1})=\Uu_{(-\alpha, 1-\lp \alpha,\lambda\rp)}$.
Condition \eqref{cd'} is satisfied if and only if
  $\lambda$ is $F$-negative (definition in \ref{induc}).

\end{proof}

\begin{proof}[Proof of Lemma \ref{lemma:compu-parab}] i. Let  $w$ be  an $F$-negative element in $\tilde\W_F$. The function
$f_{1,v}\:.\:\tau_w$ has support in $\P_F\I_F^-\hat w \I$. Since $\hat w$ satisfies \eqref{cd}, we have
$\P_F\I_F^-\hat w \I=\P_F\hat w \I=\P_F\I$. It remains to compute the value of 
$f_{1,v}\:.\:\tau_w$ at $1_\G$ (we choose the unit element $1_\G$ of $\G$  as a lift for $1\in \Dd_F$).
The proof goes through exactly as in \cite[6A.3]{Oparab} where it is written up in the case of $\mathbf G= {\rm GL}_n$.
ii.  Let  $d\in \Dd_F$.
By Lemma \ref{fact:coset}i, the $\I$-invariant function $f_{1, v}\: .\:\tau_{d}$ has support in $\P_F\hat d\I$ and it  follows from Lemma \ref{fact:coset}iii that it takes value $v$ at $\hat d$.

\end{proof}

\medskip

\subsubsection{}  Here we consider again representations with coefficients in the algebraically closed field $k$ with characteristic $p$.
We draw corollaries from Proposition  \ref{prop:invparab}.
\begin{coro} Let $F\neq x_0$ be a standard facet. If $\upsigma$ is an admissible $k$-representation  of $\M_F$ with   a central character,  then   $(\Ind_{\P_F}^\Gp\upsigma)^\I$ is a finite dimensional $\Hh_k$-module  whose 
irreducible subquotients are not supersingular.
\label{coro:supercusp}
\end{coro}

\begin{proof} The fact that   $(\Ind_{\P_F}^\Gp\upsigma)^\I$ is finite dimensional is a consequence of the admissibility of $\upsigma$.
Let $\lambda\in \X_*(\T)$ a strongly $F$-negative coweight (see Remark \ref{rema:nega}) and 
$\lambda_0\in\X_*^+(\T)$ the unique dominant coweight in its $\Wf$-orbit $\OO(\lambda)$. By Lemma \ref{theoA}
$$z_{\lambda_0}=\sum_{\lambda'\in \OO(\lambda)}\Bb_F^-(\lambda').$$
We compute the action of $z_{\lambda_0}$ on an element of the form $v\otimes 1\in  \upsigma^{\I_F^0}\otimes _{ {\Hh}_k(\M_F)^-}\Hh_k$.
We have $\Bb_F^-(\lambda)=\tau_{e^{\lambda}}$ and therefore,
$$(v\otimes 1)\Bb_F^-(\lambda)= v\otimes \tau_{e^{\lambda}}=
v \otimes j_F^-(\tau^F_{e^{\lambda}})= (v  \tau^F_{e^{\lambda}}) \otimes 1.$$
Recall that $\tau^F_{e^\lambda}=\tau^F_{\lambda(\varpi^{-1})}$ and that
$\lambda(\varpi^{-1})$ is a central element in $\M_F$. 
Therefore, $v  \tau^F_{e^{\lambda}}= \omega(\lambda(\varpi)) v$ where $\omega$ denotes the central character of $\upsigma$.
By \eqref{fact:annu}, it implies  in particular that $(v\otimes 1)\Bb_F^-(\lambda')=0$
for 
$\lambda'\in \OO(\lambda)$ distinct from $\lambda$.
We have proved that 
$z_{\lambda_0}$ acts by multiplication by  $\omega(\lambda(\varpi))\neq 0$
 on $ \upsigma^{\I_F^0}\otimes _{ {\Hh}_k(\M_F)^-}\Hh_k$  and therefore on  $(\Ind_{\P_F}^\Gp\upsigma)^\I$ by Proposition \ref{prop:invparab}. It proves the claim.
\end{proof}
\begin{coro}   Let $F$ be a standard facet.
Let ${\rm Sp}_F$ be the generalized special $k$-representation of $\Gp$
$${\rm Sp }_F=\dfrac{\Ind_{\P_F}^\Gp \:1}{\sum_{F'\neq F \subset \overline F} \Ind_{\P_{F'}}^\Gp \:1}$$ where $F'$ ranges  over the set of standard facets $\neq F$ contained in the closure of $F$.
The $\I$-invariant subspace of ${\rm Sp }_F$ 
is a finite dimensional $\Hh_k$-module  whose 
irreducible subquotients are not supersingular.\label{coro:spec}
\end{coro} 
\begin{proof}  Suppose first that $F\neq x_0$. By \cite[(18)]{GK} (which is valid with no restriction on the split group $\G$),  $({\rm Sp }_F)^\I$ is a quotient
of $(\Ind_{\P_F}^\Gp \:1)^\I$. Apply Corollary \ref{coro:supercusp}. If $F=x_0$, then the special representation in question is the trivial character of $\G$ whose  $\I$-invariant subspace is isomorphic to the trivial character of $\Hh_k$  and is not supersingular (Example \ref{triv-sign} and Lemma \ref{notss}).

\end{proof}

\subsection{On supersingular representations\label{onS}}
Let $\rho$ be a  weight  of $\K$. 
By \eqref{diag}, there is a correspondence  between the $k$-characters of $\cal H(\Gp, \rho)$ and the $k$-characters  of $\Zz^\c(\Hh_k)$,  and we will use the same letter $\zeta$ for two characters paired up by \eqref{diag}. With this notation, by the work in \ref{sec:compa},  we have a surjective morphism of representations of $\Gp$:
\begin{equation} \zeta\otimes_{\Zz^\c(\Hh_k)} \ind_\I^\Gp 1\longrightarrow  \zeta\otimes_{\cal H(\Gp, \rho)} \ind_\K^\Gp \rho.\label{quotient-univ}\end{equation}

For $\omega$ a character of the connected center  of $\Gp$, let $\zeta_\omega$ the supersingular character of $\Zz^\c(\Hh_k)$ as in
\ref{supersing-block}. Remark that the representation $\zeta_\omega \otimes_{\Zz^\c(\Hh_k)} \ind_\I^\Gp 1$  of $\Gp$ has  central character  $\omega$. \\

From now on  we suppose that the derived group of $\Gp$ is simply connected and  that $\Corps$ is a finite extension of $\mathbb Q_p$.
\begin{lemma} A character $\cal H(\Gp, \rho)\rightarrow k$   
is parametrized by the pair $(\Gp, \omega)$  in   the sense  of \cite[Proposition 4.1]{Parabind} if and only if 
it corresponds to the supersingular character $\zeta_\omega$  of $\Zz^\c(\Hh_k)$ \emph{via} \eqref{diag}.
\end{lemma}
\begin{proof}  In this proof we denote by   $\psi:\cal H(\Gp, \rho)\rightarrow k$ and 
  $\zeta: \Zz^\c(\Hh_k)\rightarrow k$ a pair of characters corresponding to each other  by \eqref{diag}.  
  Recall that $\EuScript T$ denotes the inverse Satake isomorphism \eqref{monsatake1}.
  By \cite[Corollary 4.2]{Parabind} (see also Corollary 2.19 loc.cit),  the character $\psi:\cal H(\Gp, \rho)\rightarrow k$ is parametrized by the pair $(\Gp, \omega)$
 if and only if $\psi \circ \EuScript T(\lambda)=0$ for all $\lambda\in \X_*^+(\T)$ such that 
 %$\lambda(\varpi)$ does not belong to the connected center of $\Gp$, which is equivalent to 
 $\ell(e^\lambda)\neq 0$ and if $\psi\otimes _{\cal H(\Gp, \rho)} \ind_\K^\G\rho$ has central character equal to $\omega$ (see Lemma 4.4 and its proof loc.cit). Since, for all $\lambda\in\X_*^+(\T)$, we have
$\zeta(z_\lambda)=\psi\circ \EuScript T(\lambda)$ and since
$ \psi\otimes_{\cal H(\Gp, \rho)} \ind_\K^\Gp \rho$ is a quotient of $\zeta\otimes_{\Zz^\c(\Hh_k)} \ind_\I^\Gp 1$, we have proved (using  the remark  before the statement of this lemma) that $\psi$ is parametrized by the pair $(\Gp, \omega)$ if and only if $\zeta=\zeta_\omega$.

\end{proof}

%Recall  that a character $\zeta: \cal H(\Gp, \rho)\rightarrow k$ is parameterized by a pair $(\Gp, \omega)$ for some character $\omega$ of the  connected center of $\Gp$ in   the sense  of \cite[Proposition 4.1]{Parabind}.

A smooth irreducible admissible $k$-representation  of $\G$
has a central character. 
A smooth irreducible admissible $k$-representation $\uppi$ with central character $\omega: Z\rightarrow k^\times$
   is called supersingular  with respect to $(\K,\T, \B)$  (\cite[Definition 4.7]{Parabind})
%\mag{(in the sense of \cite{Parabind} et Abe.... voir cela en détail -- aussi Lemma 4.4 de H)} 
if  for all weights $\rho$ of  $\K$,   any  map
$\ind_\K^\Gp\rho\rightarrow \uppi$ factorizes through
$$\zeta_\omega\otimes _{\cal H(\Gp, \rho)}\ind_\K^\Gp\rho\longrightarrow \uppi.$$
 Note that if the first map is zero, then the condition is trivial.  By \eqref{quotient-univ}, a supersingular representation  with  central character $\omega: Z\rightarrow k^\times$ is therefore a quotient of $ \zeta_\omega\otimes_{\Zz^\c(\Hh_k)} \ind_\I^\Gp { 1}$  and, by Definition \ref{defi:sup-id},  of 
 $$ \ind_\I^\Gp { 1}/\mathfrak I \,\ind_\I^\Gp { 1}.$$ 
 
\begin{rema}
\begin{itemize}
\item[i.]The representation  
$\ind_\I^\Gp { 1}/\mathfrak I \,\ind_\I^\Gp { 1}$
 depends only on the conjugacy class of $x_0$. It is independent of all the choices if $\mathbf G$ is of adjoint type or $\mathbf G= {\rm GL}_n$.
 \item[ii.] An irreducible admissible representation $\uppi$ of $\Gp$ is a quotient of 
 $\ind_\I^\Gp { 1}/\mathfrak I \,\ind_\I^\Gp { 1}$ if and only if
 $\uppi^\I$ contains a supersingular $\Hh_k$-module.
 Recall that  when  the root system of $\G$ is irreducible, we have proved that the notion of supersingularity  for $\Hh_k$-modules is 
  independent of all the choices made. 
 \end{itemize}
\end{rema}

\begin{theorem} If $\G= {\rm GL}_n(\mathfrak F)$ or ${\rm PGL}_n(\mathfrak F)$, a
 smooth irreducible admissible $k$-representation $\uppi$   
 %with central character $\omega$ 
 is supersingular if and only if  $\uppi^\I$ contains a supersingular $\Hh_k$-module, that it to say if and only if $\uppi$ it is a quotient of  
\begin{equation}\label{uni}\ind_\I^\Gp { 1}/\mathfrak I \,\ind_\I^\Gp { 1}. \end{equation}
\label{equiv}
\end{theorem}
\begin{proof} Let $\uppi$ be 
 a
 smooth irreducible admissible $k$-representation of $\G$   with central character $\omega$. If it is a quotient of
 $\ind_\I^\Gp { 1}/\mathfrak I \,\ind_\I^\Gp { 1}$ then it is a quotient of
  $\zeta_\omega\otimes_{\Zz^\c(\Hh_k)} \ind_\I^\Gp {1}$, and $\uppi^\I$ contains the  supersingular character $\zeta_\omega$ of $\Zz^\c(\Hh_k)$. Therefore it contains a supersingular $\Hh_k$-module.
 By Corollaries \ref{coro:supercusp} and \ref{coro:spec}, it implies that
 $\uppi$ is neither a  representation  induced from a strict parabolic subgroup of $\G$ nor  (a twist by a character of $\G$ of)   a generalized special representation. 
By \cite[Theorem 1.1]{Parabind} that classifies all smooth irreducible admissible $k$-representation of $\G$,  we conclude by elimination that
 the representation $\uppi$  is supersingular.  
 \end{proof}

The results of \cite{Parabind}  have been generalized to 
the case of  a $\mathfrak F$-split connected reductive group $\G$ in \cite{abe}: the classification of the  smooth irreducible admissible   representations  of $\G$ is quite similar to the case of  $ {\rm GL}_n(\mathfrak F)$  (expect for a certain subtlety when the root system of $\G$ is not irreducible).
Based on this classification and on Corollaries \ref{coro:supercusp} and \ref{coro:spec}, N. Abe confirmed that  the space of $\I$-invariant vectors of a nonsupersingular representation does not  contain any supersingular $\Hh_k$-module. 
Therefore, Theorem \ref{equiv}   is true  for a general split group with simply connected derived subgroup.

%In the case where the root system of $\Gp$ is not irreducible there is a technical variation and an extra case to consider in the classification of the irreducible nonsupercuspidal  representations.
%N. Abe confirmed that the computation  of the $\I$-invariant subspace in the extra case yields an analogous result, namely this subspace does not contain any supersingular module.
%Therefore, Theorem \ref{equiv}   is true 

%\theendnotes
\end{document}